\documentclass{compositio}
\usepackage[mathscr]{eucal}

\usepackage{amsmath}
\usepackage{amssymb}

\newcommand{\F}{\mathbb{F}} 

\def\RCS$#1: #2 ${\expandafter\def\csname RCS#1\endcsname{#2}}
\RCS$Revision: 132 $
\RCS$Date: 2010-08-31 07:25:03 -0700 (Tue, 31 Aug 2010) $

\usepackage{amscd}
\usepackage{tabularx}

\usepackage[all,cmtip,line]{xy}

\DeclareMathOperator{\Frob}{Frob}

\newcommand{\To}{\longrightarrow}
\newcommand{\into}{\hookrightarrow}

\newcommand{\isoto}{\stackrel{\sim}{\To}}
\newcommand{\p}{\mathfrak{p}}

\newcommand{\M}{\mathcal{M}}
\newcommand{\N}{\mathcal{N}}

\newcommand{\bigO}{\mathcal{O}}

\newcommand{\Z}{\mathbb{Z}}
\newcommand{\A}{\mathbb{A}}
\newcommand{\Q}{\mathbb{Q}}
\newcommand{\T}{\mathbb{T}}

\newcommand{\C}{\mathbb{C}}

\newcommand{\rhobar}{\overline{\rho}}

\newcommand{\Gal}{\operatorname{Gal}}
\newcommand{\GL}{\operatorname{GL}}
\newcommand{\PGL}{\operatorname{PGL}}

\newcommand{\Qbar}{\overline{\Q}}
\newcommand{\Qp}{\Q_p}
\newcommand{\Fp}{\F_p}

\newcommand{\Qpbar}{\overline{\Q}_p}
\newcommand{\Fpbar}{\overline{\F}_p}

\newcommand{\BrMod}{\operatorname{BrMod}}

\newcommand{\End}{\operatorname{End}}

\newcommand{\Aut}{\operatorname{Aut}}

\newcommand{\Fil}{\mathrm{Fil}}

\newcommand{\Nm}{\mathrm{N}}
\newcommand{\ur}{\mathrm{ur}}

\newcommand{\JL}{\mathrm{JL}}
\newcommand{\LL}{\mathrm{LL}}
\newcommand{\St}{\mathrm{St}}

\newcommand{\om}{\widetilde{\omega}}
\newcommand{\omt}{\om_2}
\newcommand{\OO}{\calO}
\newcommand{\m}{\frakm}
\newcommand{\calO}{{\mathcal O}}
\newcommand{\frakm}{\mathfrak{m}}
\newcommand{\Zp}{\Z_{p}}
\newcommand{\st}{\mathrm{st}}

\newcommand{\G}{{\mathcal G}}
\newcommand{\univ}{\operatorname{univ}}
\newcommand{\dd}{\operatorname{dd}}

\usepackage{amsthm}
\numberwithin{equation}{section}

\newtheorem{thm}[equation]{Theorem}

\newtheorem{cor}[equation]{Corollary}

\newtheorem{lem}[equation]{Lemma}
\newtheorem{prop}[equation]{Proposition}
 \theoremstyle{definition}
\newtheorem{defn}[equation]{Definition} %\theoremstyle{remark}
\newtheorem{rem}[equation]{Remark} %\numberwithin{equation}{subsection}
\newtheorem{remark}[equation]{Remark} %\numberwithin{equation}{thm}

\begin{document}
\title{Serre weights for
  quaternion algebras}
\author{Toby Gee} \email{tgee@math.harvard.edu} \address{Department of
  Mathematics, Harvard University}
\author{David Savitt} \email{savitt@math.arizona.edu}
\address{Department of Mathematics, University of Arizona}
\thanks{The first author was partially
  supported by NSF grant DMS-0841491, and the second author was
partially supported by NSF grants DMS-0600871 and DMS-0901049.}

\begin{abstract}We study the possible weights of an irreducible
  2-dimensional
 mod~$p$  representation of $\Gal(\overline{F}/F)$
which is modular in the sense of that it comes from an automorphic
form on a definite quaternion algebra with centre $F$ which is ramified at all
places dividing $p$, where $F$
is a totally real field. In most
cases we determine the precise list of possible weights; in the
remaining cases we determine the possible weights up to a short and
explicit list of exceptions.
\end{abstract}

\maketitle %\tableofcontents
\section{Introduction}%\versioninfo

Let $F$ be a totally real field and let $p$ be a prime number.  In
this paper we formulate, and largely prove, an analogue of
the weight part of Serre's conjecture \cite{ser87} for automorphic forms
on quaternion algebras over $F$ which are ramified at all places dividing
$p$. 
%The study of the possible weights of a mod $p$
%modular Galois representation was initiated by Serre in his famous
%paper \cite{ser87}. This proposed a concrete conjecture (``the
%weight part of Serre's conjecture'') relating the weights to the
%restriction of the Galois representation to an inertia subgroup at
%$p$. This conjecture was resolved (at least for $p>2$) by work of
%Coleman-Voloch, Edixhoven and Gross (see \cite{MR1176206}).

In recent years, a great deal of attention has been given to the problem of generalizing the weight part of Serre's
conjecture beyond  the case of $\GL(2,\Q)$, beginning with the seminal paper
\cite{bdj}
which considered the situation for Hilbert modular forms.  Let $G_F$
denote the absolute Galois group of $F$; then to any irreducible modular
representation $$\rhobar:G_F\to\GL_2(\Fpbar)$$ there is associated a
set of weights $W(\rhobar)$, the set of weights in which $\rhobar$ is
modular (see section \ref{sec: notation} for the definitions of
weights and of what it means for $\rhobar$ to be modular of a certain
weight). Under the assumption that $p$ is unramified in $F$, the paper
\cite{bdj} associated to $\rhobar$ a set of weights $W^{?}(\rhobar)$, and
conjectured that $W^?(\rhobar)=W(\rhobar)$.  Schein
\cite{scheinramified} subsequently proposed a generalisation  that, in the
tame case (where the restrictions of $\rhobar$ to inertia subgroups at
places dividing $p$ are semisimple),  removes the
restriction that $p$ be unramified in $F$.  When $p$ is either unramified or totally
ramified in $F$ many cases of these
conjectures have been proved, in \cite{gee053} and \cite{geesavitttotallyramified}
respectively, but the general case has so far been out of reach.

%The set $W^?(\rhobar)$ depends only on the restrictions of $\rhobar$
%to inertia subgroups at places dividing $p$. In the case that these
%restrictions are tamely ramified, the conjecture is completely
%explicit, while in the general case the set depends on some rather
%delicate questions involving extensions of crystalline characters.

As far as we know, there is no corresponding conjecture in the
literature for automorphic forms on quaternion algebras that are
ramified at $p$ (although the results in the case $F=\Q$ are easily deduced
from the discussion in section 4 of \cite{MR1802794}). We specify a
conjectural set of weights $W^?(\rhobar)$, depending only on the
restrictions of $\rhobar$ to decomposition subgroups at places dividing
$p$. In the case that these restrictions are semisimple, the
conjecture is completely explicit (and depends only on the
restrictions to inertia subgroups).  In the general case the set $W^?(\rhobar)$
is defined in terms of the existence of certain potentially
Barsotti-Tate lifts of specific type, and so depends on some rather
delicate questions involving extensions of crystalline characters. In
fact, we always make a definition in terms of the existence of certain potentially
Barsotti-Tate lifts, and in the semisimple case we make this
description explicit by means of calculations with Breuil modules and
strongly divisible modules.

We assume that $p$ is odd, and that $\rhobar|_{G_{F(\zeta_p)}}$ is
irreducible. We make a mild additional assumption if $p=5$.  All of
these restrictions are imposed by our use of the modularity lifting
theorems of \cite{kis04} (or rather, by their use in
\cite{gee061}). Under these assumptions, we are able to prove that if
for each place $v|p$, $\rhobar|_{G_{F_v}}$ is not of the form
$\left(\begin{smallmatrix}\psi_1 & *\\ 0
    &\psi_2\end{smallmatrix}\right)$ with $\psi_1/\psi_2$ equal to the
mod $p$ cyclotomic character, then $W^?(\rhobar)=W(\rhobar)$. In the
exceptional cases we establish that $W(\rhobar)\subset W^?(\rhobar)$,
with equality up to a short list of possible exceptions (for example,
in the case that there is only one place of $F$ above $p$ there is
only one exception). Our techniques are analogous to those of
\cite{gee053} and \cite{geesavitttotallyramified}. As in those papers,
the strategy is to construct modular lifts of
$\rhobar$ which are potentially semistable of specific type, using
the techniques of Khare-Wintenberger, as explained in
\cite{gee061}.

The significant advantage in the present situation over those
considered in \cite{geesavitttotallyramified}  and \cite{gee053}  is that the property of being modular of a
specific weight corresponds exactly to the property of having a lift
of some specific type (in the case considered in \cite{gee053} this
correspondence was considerably weaker). Accordingly, we have no regularity
assumption on the weights,  we do not have to use Buzzard's 
``weight cycling'' technique, and especially  we do not need
to make any restriction on the splitting behaviour of $p$ in $F$.

In the case that the restrictions of $\rhobar$ to decomposition groups
at places dividing $p$ are all semisimple, we establish an explicit
description of $W^?(\rhobar)$ by a computation in two stages. In one
step we make use of Breuil modules with descent data, in the same
style as analogous computations in the literature; our calculations
are more complicated than those made in the past, however, as
we have no restrictions on the ramification or inertial degrees of our
local fields.

For the second step, we have to exhibit potentially Barsotti-Tate
lifts of the appropriate types. Writing down such lifts is rather
non-trivial. We accomplish this by means of an explicit construction
of corresponding strongly divisible modules; again, these calculations
are more complicated than those in the literature, because we make no
restrictions on the ramification or inertial degrees of our local
fields.

We also note that while we work throughout with definite quaternion
algebras, it should not be difficult to extend our results to
indefinite algebras; one needs only to establish the analogue of Lemma
\ref{432} (see
for example the proof of Lemma 6 of \cite{MR1272977} for the case $F=\Q$).

We now detail the outline of the paper. In section \ref{sec:
notation} we give our initial definitions and notation. In
particular, we introduce spaces of algebraic modular forms on
definite quaternion algebras, and we explain what it means for
$\rhobar$ to be modular of a specific weight. 

In section \ref{sec:tame lifts} we explain which tame lifts we will
need to consider, and the relationship between the existence of
modular lifts of specified types and the property of being modular
of a certain weight. This amounts to recalling certain concrete
instances of the local Langlands and Jacquet-Langlands correspondences for $\GL_2$ and
local-global compatibility. All of this material is
completely standard.

Section \ref{sec: necessity} begins the work of establishing an
explicit description of $W^?(\rhobar)$, by finding necessary
conditions for the existence of potentially Barsotti-Tate lifts of
particular type, via calculations with Breuil modules. The sufficiency
of these conditions is established in section \ref{sec:local lifts},
by writing down explicit strongly divisible modules. Both sections
make use of a lemma relating the type of the lifts to the descent data
on the Breuil modules and strongly divisible modules, which is
established in section \ref{sec:types-and-descent}.

These local calculations are summarised in section \ref{sec: explicit
  local description of the weights}. Finally, in section \ref{main
  results} we prove our main theorem.

We are grateful to the anonymous referee for reading the paper
carefully and providing a number of valuable suggestions.   The second author thanks the MPIM for its hospitality.

\section{Notation and assumptions}\label{sec: notation}Let $p$ be an odd prime. Fix an
algebraic closure $\Qbar$ of $\Q$, an algebraic closure  $\Qpbar$ of
$\Qp$, and an embedding $\Qbar\into\Qpbar$. We will consider all
finite extensions of $\Q$ (respectively $\Qp$) to be contained in
$\Qbar$ (respectively $\Qpbar$). If $K$ is such an extension, we let
$G_{K}$ denote its absolute Galois group $\Gal(\overline{K}/K)$. Let $F$ be a totally real field. Let $\rhobar:G_{F}\to\GL_{2}(\Fpbar)$ be a continuous
representation. Assume from now on that
$\rhobar|_{G_{F(\zeta_{p})}}$ is absolutely irreducible. If $p=5$ and
the projective image of $\rhobar$ is isomorphic to $\PGL_{2}(\F_{5})$,
assume further that $[F(\zeta_{5}):F]=4$. We normalise the
isomorphisms of local class field theory so that a uniformiser
corresponds to a geometric Frobenius element.

We wish to discuss the Serre weights of $\rhobar$ for quaternion
algebras ramified at all places dividing $p$. We choose to work
with totally definite quaternion algebras. We recall the basic
definitions and results that we need.

Let $D$ be a quaternion algebra with center $F$ which is ramified at
all infinite places of $F$ and at a set $\Sigma$ of finite places
which contains all primes dividing $p$. Fix a maximal order $\bigO_D$
of $D$ and for each finite place $v\notin\Sigma$ fix an isomorphism
$\bigO_{D,v}:=(\bigO_D)_v\isoto M_2(\bigO_{F_v})$. For any finite place $v$ let
$\pi_{v}$ denote a uniformiser of $F_v$.

Let $U=\prod_v U_v\subset (D\otimes_F\mathbb{A}^f_F)^\times$ be a
compact open subgroup, with each $U_v\subset
\bigO_{D,v}^\times$. Furthermore, assume that $U_v=\bigO_{D,v}^\times$
for all $v\in\Sigma$.

Take $A$ a topological $\Z_p$-algebra. For each place $v|p$ fix a
continuous representation $\sigma_v:U_{v}\to\Aut(W_{v})$ with $W_{v}$
a finite $A$-module. Let $\sigma$ denote the representation
$\otimes_{v|p}\sigma_v$ of $U_p:=\prod_{v|p}U_v$, acting on
$W_\sigma:=\otimes_{v|p}W_v$. We regard $\sigma$ as a representation
of $U$ in the obvious way (that is, we let $U_v$ act trivially if
$v\nmid p$). Fix also a character
$\psi:F^\times\backslash(\mathbb{A}^f_F)^\times\to A^\times$ such that
for any finite place $v$ of $F$, $\sigma|_{U_v\cap\bigO_{F_v}^\times}$
is multiplication by $\psi^{-1}$. Then we can think of $W_\sigma$ as a
$U(\mathbb{A}_F^f)^\times$-module by letting $(\mathbb{A}_F^f)^\times$
act via $\psi^{-1}$.

Let $S_{\sigma,\psi}(U,A)$ denote the set of continuous
functions $$f:D^\times\backslash(D\otimes_F\mathbb{A}_F^f)^\times\to
W_\sigma$$ such that for all $g\in (D\otimes_F\mathbb{A}_F^f)^\times$
we have $$f(gu)=\sigma(u)^{-1}f(g)\text{ for all }u\in
U,$$ $$f(gz)=\psi(z)f(g)\text{ for all
}z\in(\mathbb{A}_F^f)^\times.$$We can write
$(D\otimes_F\mathbb{A}_F^f)^\times=\coprod_{i\in I}D^\times
t_iU(\mathbb{A}_F^f)^\times$ for some finite index set $I$ and some
$t_i\in(D\otimes_F\mathbb{A}_F^f)^\times$. Then we
have $$S_{\sigma,\psi}(U,A)\isoto\oplus_{i\in
  I}W_\sigma^{(U(\mathbb{A}_F^f)^\times\cap t_i^{-1}D^\times
  t_i)/F^\times},$$the isomorphism being given by the direct sum of
the maps $f\mapsto f(t_i)$. From now on we make the following
assumption:$$\text{For all
}t\in(D\otimes_F\mathbb{A}_F^f)^\times\text{ the group
}(U(\mathbb{A}_F^f)^\times\cap t^{-1}D^\times t)/F^\times=1.$$ One can
always replace $U$ by a subgroup (satisfying the above assumptions,
and  without changing $U_p$)
for which this holds (\emph{cf.} section 3.1.1 of \cite{kis07}). Under this
assumption $S_{\sigma,\psi}(U,A)$ is a finite $A$-module,
and the functor $W_\sigma\mapsto S_{\sigma,\psi}(U,A)$ is exact in
$W_\sigma$.

We now define some Hecke algebras. Let $S$ be a set of finite places
containing $\Sigma$ and the primes $v$ of $F$ such that $U_v
\neq \bigO_{D,v}^\times$.  Let
$\T^{\operatorname{univ}}_{S,A}=A[T_v,S_v]_{v\notin S}$ be the
commutative polynomial ring in the formal variables $T_v$,
$S_v$. Consider the left action of $(D\otimes_F\mathbb{A}_F^f)^\times$
on $W_\sigma$-valued functions on $(D\otimes_F\mathbb{A}_F^f)^\times$
given by $(gf)(z)=f(zg)$. Then we make $S_{\sigma,\psi}(U,A)$ a
$\T^{\operatorname{univ}}_{S,A}$-module by letting $S_v$ act via the
double coset $U\bigl(
\begin{smallmatrix}
    \pi_{v}&0\\0&\pi_{v}
\end{smallmatrix}
\bigr)U$ and $T_v$ via $U\bigl(
\begin{smallmatrix}
    \pi_{v}&0\\0&1
\end{smallmatrix}
\bigr)U$. These are independent of the choices of $\pi_{v}$. We will write $\T_{\sigma,\psi}(U,A)$ or $\T_{\sigma,\psi}(U)$ for the image of $\T^{\operatorname{univ}}_{S,A}$ in $\End S_{\sigma,\psi}(U,A)$.

Let $\mathfrak{m}$ be a maximal ideal of
$\T^{\operatorname{univ}}_{S,A}$. We say that $\mathfrak{m}$ is \emph{in the
support of $(\sigma,\psi)$} if $S_{\sigma,\psi}(U,A)_\mathfrak{m}\neq
0$. Now let $\bigO$ be the ring of integers in $\Qpbar$, with residue
field $\F=\Fpbar$, and suppose that $A=\bigO$ in the above discussion,
and that $\sigma$ has open kernel and is free as an $\bigO$-module. Consider a maximal ideal
$\mathfrak{m}\subset\T^{\operatorname{univ}}_{S,\bigO}$ with residue
field $\F$ which is in the support of $(\sigma,\psi)$. Then there is a semisimple Galois
representation $\overline{\rho}_\mathfrak{m}:G_F\to\GL_2(\F)$
associated to $\mathfrak{m}$ which is characterised up to conjugacy by
the property that if $v\notin S$ then
$\rhobar_\mathfrak{m}|_{G_{F_v}}$ is unramified, and if $\Frob_v$ is
an arithmetic Frobenius at $v$ then the characteristic polynomial of
$\overline{\rho}_\mathfrak{m}(\Frob_v)$ is the image of $X^2-T_vX+S_v\mathbf{N}v$ in $\F[X]$.

We have the following basic lemma.
\begin{lem}
  \label{432}Let
  $\psi:F^{\times}\backslash(\A_{F}^f)^{\times}\to\bigO^{\times}$ be a
  continuous character, and write $\overline{\psi}$ for the composite
  of $\psi$ with the projection $\bigO^{\times}\to \F^{\times}$. Fix a
  representation $\sigma'$ of $U_p$ on a finite free $\bigO$-module
  $W_{\sigma'}$, and an irreducible representation $\sigma$ of $U_p$ on a
  finite free $\F$-module $W_{\sigma}$. Suppose that we have
  $\sigma'|_{U_v\cap\bigO_{F_v}^\times}=\psi^{-1}|_{U_v\cap\bigO_{F_v}^\times}$
  and
  $\sigma|_{U_v\cap\bigO_{F_v}^\times}=\overline{\psi}^{-1}|_{U_v\cap\bigO_{F_v}^\times}$
  for all finite places $v$.

    Let $\mathfrak{m}$ be a maximal ideal of $\mathbb{T}_{S,\bigO}^{\univ}$.

    Suppose that $W_{\sigma}$ occurs as a $U_p$-module subquotient of $W_{\sigma'}\otimes\F$. If $\mathfrak{m}$ is in the support of $(\sigma,\overline{\psi})$, then $\mathfrak{m}$ is in the support of $(\sigma',\psi)$.

    Conversely, if $\mathfrak{m}$ is in the support of
        $(\sigma',\psi)$, then $\mathfrak{m}$ is in the support of
        $(\sigma,\overline{\psi})$ for some irreducible $U_p$-module
        subquotient $W_{\sigma}$ of $W_{\sigma'}\otimes\F$.
\end{lem}
\begin{proof}
  The first part is proved just as in Lemma 3.1.4 of \cite{kis04}, and
  the second part follows, for example from Proposition 1.2.3 of
  \cite{as862}, or from a basic commutative algebra argument.
\end{proof}

We are now in a position to define what it means for a
representation to be modular of some weight. Let $v|p$ be a place of
$F$, so that $U_v=\bigO_{D,v}^\times$. Let $\sigma_v$ be an irreducible
$\F$-representation of $U_v$. Note that if $\Pi_v$ is a uniformiser of
$\bigO_{D,v}$, then $k_{2,v}:=\bigO_{D,v}/\Pi_v$ is a finite field, a
quadratic extension of the residue field $k_v$ of $F_v$. The kernel of
the reduction map $U_v\to k_{2,v}^\times$ is a pro-$p$ group, so
$\sigma_v$ factors through this kernel, and is a
representation of the finite abelian group $k_{2,v}^\times$. It is
therefore one-dimensional. Let $\sigma=\otimes_{v|p}\sigma_v$, which
we will regard as an $\bigO$-module via the natural map $\bigO\to\F$.
\begin{defn}
    We say that $\overline{\rho}$ is \emph{modular of weight $\sigma$} if
        for some $D$, $S$, $U$, $\psi$, and $\mathfrak{m}\subset\T^{\operatorname{univ}}_{S,\bigO} $ as above,
        we have $S_{\sigma,\psi}(U,\bigO)_\mathfrak{m}\neq 0$ and $\overline{\rho}_\mathfrak{m}\cong\overline{\rho}$.
\end{defn}[Here $\rhobar_{\mathfrak{m}}$ is characterised as above,
and exists by Lemma \ref{432} and the remarks above.]
Write
$W(\rhobar)$ for the set of weights $\sigma$ for which $\rhobar$ is 
modular of weight $\sigma$. 
Assume from now on that $\rhobar$ is modular of some weight, and fix
$D$, $S$, $U$, $\psi$, $\mathfrak{m}$ as in the definition.

We also have the following useful lemma, which was first observed by
Serre in the case $F=\Q$ (see remark (11) in Serre's letter to Tate in
\cite{MR1418297}). For each place $v|p$ we let $q_v$ denote the order
of the residue field $k_v$ of $F_v$. If $\sigma_v$ is an irreducible
$\F$-representation of $\bigO_{D,v}^\times$, then by the remarks above
it is a character of $k_{2,v}^\times$. Thus $\sigma_v^{q_v}$ is
another irreducible $\F$-representation, and
$\sigma_v^{q_v^2}=\sigma_v$. Serre observed that the set $W(\rhobar)$
is preserved by this operation. This is essentially a consequence of
the structure of $\bigO_{D,v}$. Let $K_{2,v}=W(k_{2,v})[1/p]$, a
subfield of $D_v$. Note that there is a choice of uniformiser $\Pi_v$
of $D_v$ with the property that conjugation by $\Pi_v$ preserves
$K_{2,v}$, and acts on it via a non-trivial involution. In particular,
the induced action on $k_{2,v}$ is via the $q_v$-th power map.

\begin{lem}\label{lem: set of weights is preserved by local Frobenius}
  Let $v$ be a place of $F$ dividing $p$, and let $\sigma$ be a weight
  as above. Let $\sigma'=\sigma_v^{q_v}\otimes_{w|p,w\neq
    v}\sigma_w$. Then $\rhobar$ is modular of weight $\sigma$ if and
  only if it is modular of weight $\sigma'$.
\end{lem}
\begin{proof}
  It suffices to exhibit a bijection \[\theta:S_{\sigma,\psi}(U,\F)\to
  S_{\sigma',\psi}(U,\F)\]which commutes with the action of
  $\mathbb{T}_{S,\bigO}^{\univ}$. Let $\Pi\in (D\otimes_F\A^{f})^\times$ be
  trivial away from $v$, and equal to $\Pi_v$ at $v$, where $\Pi_v$ is
  as in the previous paragraph. Then we define $\theta$ by \[(\theta
  f)(x):=f(x\Pi).\]It is straightforward to check that this map has
  the required properties; the key point is that if $u\in U$, then \begin{align*}(\theta f)(xu)&=f(xu\Pi)\\
    &= f(x\Pi (\Pi^{-1}u\Pi))\\ &=
    \sigma(\Pi^{-1}u\Pi)^{-1}f(x\Pi)\\&= \sigma'(u)^{-1}f(x\Pi)\\&=
    \sigma'(u)^{-1}(\theta f)(x).\end{align*}
\end{proof}

\section{Weights are controlled by lifts of tame type}\label{sec:tame
  lifts}  Continue to let $v$ be a place of $F$ that divides $p$.  We distinguish two types of
irreducible $\F$-representations $\sigma_v$ of $U_v$. Recall that any
such representation is $1$-dimensional, and factors through
$k_{2,v}^\times$, with $k_{2,v}$ a quadratic extension of $k_v$.

\begin{defn}
  We say that $\sigma_v$ is of \emph{type~I} if it does not factor through
  the norm $k_{2,v}^\times\to k_v^\times$. Otherwise, we say that it
  is of \emph{type~II}.
\end{defn}

We now recall some facts about the local Langlands and local
Jacquet-Langlands correspondences. Let $K$ be a finite extension of
$\Qp$, let $L$ be an unramified quadratic extension of $K$, and let $D$
be a nonsplit quaternion algebra over $K$. Consider $L$ as a subfield
of $D$. Let $k$ be the residue field of $K$, of cardinality $q$. If $\pi$ is an irreducible
admissible ($\C$- or $\Qpbar$-valued) representation of $D^\times$, we let $\JL(\pi)$ be the
corresponding representation of $\GL_2(K)$. If $\pi$ is an irreducible
admissible representation of $\GL_2(K)$, we let $\LL(\pi)$ denote the
corresponding representation of the Weil group $W_K$ of $K$. Let
$N_D:D^\times\to K^\times$ be the reduced norm. As usual we identify
characters of $L^\times$ or $K^\times$ with characters of the
corresponding Weil groups via local class field theory. If $\chi$ is a
character of $L^\times$ which does not factor through the norm to
$K^\times$, we denote the corresponding supercuspidal representation
of $\GL_2(K)$ by $W_\chi$.

\begin{itemize}
\item If $\chi$ is a character of $K^\times$, then
  $\JL(\chi\circ N_D)=(\chi\circ\det)\otimes \St$, where $\St$ is the Steinberg representation.
\item Suppose that $\chi$ is a character of $L^\times$ of
  conductor $1$. Then
  $\LL(\pi)|_{I_K}\cong\chi|_{I_K}\oplus\chi|_{I_K}^q$ if and only if $\pi=W_{\chi'}$ for some unramified twist
  $\chi'$ of $\chi$. [See section A.3.2 of \cite{bm}.]
\item If $\chi$ is a character of $L^\times$ of
  conductor $1$, then $\JL^{-1}(W_\chi)$ is 2-dimensional. Furthermore,
  $\JL^{-1}(W_\chi)|_{\bigO_L^\times}\cong\chi|_{\bigO_L^\times}\oplus\chi|_{\bigO_L^\times}^q$. [See
  section 7 of \cite{MR1059954}.]
\end{itemize}

We now recall some definitions relating to potentially semistable
lifts of particular type. We use the conventions of \cite{sav04}.

\begin{defn}\label{defn:barsotti tate lifts} Let $\tau_v$ be an inertial
  type. We say that a lift $\rho:G_{F_v}\to\GL_2(\Qpbar)$ of $\rhobar|_{G_{F_{v}}}$ is
  \emph{parallel potentially Barsotti-Tate} (respectively \emph{parallel potentially semistable}) of type $\tau_v$ if $\rho$ is potentially
  Barsotti-Tate (respectively potentially semistable with all Hodge-Tate
  weights equal to 0 or 1), has determinant a finite order character of order
  prime to $p$ times the cyclotomic character, and the corresponding
  Weil-Deligne representation, when restricted to $I_{F_{v}}$, is
  isomorphic to $\tau_v$. \end{defn}
Note that for a two-dimensional de Rham representation with all
Hodge-Tate weights equal to $0$ or $1$, the condition that all pairs of labeled
Hodge-Tate weights are $\{0,1\}$ is equivalent to the condition that 
the determinant is the product of the cyclotomic character, a finite
order character, and an unramified character; the condition of being
parallel is slightly stronger than this.

If $\sigma_v$ is an irreducible $\F$-representation of $U_v$, we will consider the
inertial type of $I_{F_v}$ given by
$\widetilde{\sigma}_v\oplus\widetilde{\sigma}_v^{q_v}$, where a tilde denotes
a Teichm\"{u}ller lift (considered as a
representation of $I_{F_v}$ via local class field theory). 

% The following result follows from Lemmas \ref{432} and \ref{lem: set of weights is preserved by local Frobenius}, the
% Jacquet-Langlands correspondence, the remarks made above about the
% local Langlands and Jacquet-Langlands correspondences, and the compatibility of the local
% and global Langlands correspondences at places dividing $p$ (see
% \cite{kis06}).

\begin{lem}\label{lem:typesversusweights}  $\rhobar$ is modular of weight
  $\sigma=\otimes_{v|p}\sigma_v$ if and only if $\rhobar$ lifts to a
  modular Galois representation $\rho:G_F\to\GL_2(\Qpbar)$ which for all places $v|p$ is parallel
  potentially Barsotti-Tate of
  type $\widetilde{\sigma}_v\oplus\widetilde{\sigma}_v^{q_v}$ at $v$ if $\sigma_v$ is of type~I, and is parallel potentially
  semistable of type $\widetilde{\sigma}_v\oplus\widetilde{\sigma}_v^{q_v}$ at $v$, but not potentially crystalline, if
  $\sigma_v$ is of type~II.

\end{lem}
\begin{proof}
  We first tackle the only if direction. If $\sigma_v$ is of type~I
  then we choose an arbitrary extension of $\widetilde{\sigma}_v$ to a character of
  $F_{v,2}^\times$, where $F_{v,2}$ is the unramified quadratic
  extension of $F_v$, and if  $\sigma_v$ is of type~II
  then we choose an arbitrary extension of $\widetilde{\sigma}_v$ to a character of
  $F_{v}^\times$. We continue to denote these extensions by $\widetilde{\sigma}_v$. We apply Lemma \ref{432},
  with $\m$ and $\overline{\psi}$ chosen such that $\rhobar_\m\cong\rhobar$ and
  $S_{\sigma,\overline{\psi}}(U,\F)_\m\neq 0$ (such a
  maximal ideal exists by the assumption that $\rhobar$ is modular of
  weight $\sigma$),
  and \[\sigma'=\otimes_{v|p}\sigma'_v\] where
  \[\sigma'_v=(\JL^{-1}(W_{\widetilde{\sigma}_v}))|_{\bigO_{D,v}^\times}\]
  if $\sigma_v$ is of type~I, and \[\sigma'_v=\widetilde{\sigma}_v\]if
  $\sigma_v$ is of type II. We take the $\psi$ of Lemma
  \ref{432} to be the Teichm\"uller lift of $\overline{\psi}$. The correspondence between the algebraic modular
  forms considered in section \ref{sec: notation} and automorphic
  representations of $D^\times$ is explained in section 3.1.14 of
  \cite{kis04} (there is a running assumption in that paper that $D$
  is split at all places dividing $p$, but it is not needed in this
  discussion, and if one sets the representation
  $W_{\tau^{\operatorname{alg}}}$ of \emph{loc. cit.} to be the
  trivial representation the discussion goes through immediately in
  our case), and we see that after choosing an isomorphism $\Qpbar\isoto\C$ there is an automorphic representation
  $\pi$ of $D^\times$ whose weight is the trivial representation,
  whose Hecke polynomials at unramified places lift the characteristic
  polynomials of the corresponding Frobenius elements for $\rhobar$,
  and such that for each place $v|p$, $\pi_v$ is either
  $\JL^{-1}(W_\chi)$ for $\chi$ an unramified twist of
  $\widetilde{\sigma}_v$ if $\sigma_v$ is of type~I, or an unramified
  twist of $\widetilde{\sigma}_v$ if $\sigma_v$ is of type~II. [To see
  this in the case that $\sigma_v$ has type~I, note that if
  $\pi_v|_{\bigO_{D,v}^\times}$ contains
  $(\JL^{-1}(W_{\widetilde{\sigma}_v}))|_{\bigO_{D,v}^\times}$, then
  the conductor of $\pi_v$ is at most the conductor of
  $(\JL^{-1}(W_{\widetilde{\sigma}_v}))|_{\bigO_{D,v}^\times}$. By the
  results recalled in section 7 of \cite{MR1059954}, we see that
  $\pi_v$ must be of the form $W_\chi$ for $\chi$ a character of
  $F_{2,v}^\times$ of conductor $1$. Since (for example by the
  character formulae in section 7 of \cite{MR1059954})
  $W_{\widetilde{\sigma}_v}\cong W_{\widetilde{\sigma}_v^{q_v}}$, the
  result follows from the third bullet point above.]

  Applying the Jacquet-Langlands correspondence (Theorem 16.1 of
  \cite{MR0401654}) we see that there is an automorphic representation
  $\pi'$ of $\GL_2(\A_F)$ with the same infinitesimal character as the
  trivial representation, whose Hecke polynomials at unramified places
  lift the characteristic polynomials of the corresponding Frobenius
  elements for $\rhobar$, and such that for each place $v|p$, $\pi_v$
  is either $W_\chi$ for $\chi$ an unramified twist of
  $\widetilde{\sigma}_v$ if $\sigma_v$ is of type~I, or an unramified
  twist of $(\widetilde{\sigma}_v\circ\det)\otimes \St$ if $\sigma_v$
  is of type~II. The compatibility of the local
and global Langlands correspondences at places dividing $p$ (see
\cite{kis06}), and the results on the form of the local Langlands
correspondence recalled above, show that the Galois representation
corresponding to $\pi'$ gives a representation of the required form
(note that the Galois representation has determinant
$\psi\epsilon$, so is indeed parallel).

For the converse, we may reverse the above argument, and we see that
Lemma \ref{432} guarantees that $\rhobar$ is modular of a weight
$\mu=\otimes_{v|p}\mu_{v}$, where for each $v|p$ if $\sigma_v$ is of type~II
then $\mu_{v}=\sigma_v$, and if $\sigma_v$ is of type~I then
$\mu_{v}=\sigma_v$ or $\sigma_v^{q_v}$. The result then follows from
Lemma \ref{lem: set of weights is preserved by local Frobenius}.
\end{proof}

This motivates the following definition of $W^?(\rhobar)$. 

\begin{defn}\label{defn: the set of conjectural weights}
  For each place $v|p$, let $W^?(\rhobar|_{G_{F_v}})$ denote the set of
  $\sigma_v$ of type~I such that $\rhobar|_{G_{F_v}}$ has a
  parallel potentially Barsotti-Tate lift of type $\widetilde{\sigma}_v\oplus\widetilde{\sigma}_v^{q_v}$, together with the
  set of $\sigma_v$ of type~II such that $\rhobar|_{G_{F_v}}$ has a parallel
  potentially semistable lift of type $\widetilde{\sigma}_v\oplus\widetilde{\sigma}_v^{q_v}$ which is not
  potentially crystalline. Let $W^?(\rhobar)$ be the set of weights
  $\sigma=\otimes_{v|p}\sigma_v$ with $\sigma_v\in
  W^?(\rhobar|_{G_{F_v}})$ for all $v|p$.
\end{defn}

Note that by Lemma \ref{lem:typesversusweights} we have
$W(\rhobar)\subset W^?(\rhobar)$. We will
prove under a mild hypothesis that $W(\rhobar)=W^?(\rhobar)$ in
section \ref{main results}. In the intervening sections we will give
an explicit description of $W^?(\rhobar)$ in the case that
$\rhobar|_{G_{F_v}}$ is semisimple for each $v|p$. It is already
possible to see that weights of type~II are rather rare.

\begin{lem}\label{lem:weight of type II implies cyclo}If $\rhobar$ is
  modular of weight $\sigma=\otimes_{v|p}\sigma_v$, and $\sigma_v$ is
  of type~II,
  then \[\rhobar|_{I_{F_v}}\cong\sigma_v\begin{pmatrix}\epsilon
    & *\\ 0 & 1\end{pmatrix}\]where $\sigma_v$ is regarded as a character of
  $I_{F_v}$ via local class field theory, and $\epsilon$ is the
  cyclotomic character.
  
\end{lem}
\begin{proof}
  This follows from Lemma \ref{lem:typesversusweights}, and the
  well-known fact that 2-dimensional semistable non-crystalline
  $p$-adic representations with all pairs of labeled Hodge-Tate
  weights equal to $\{0,1\}$ are unramified twists of an extension of
  the trivial character by the cyclotomic character.
\end{proof}
\section{Necessary conditions}\label{sec: necessity}

\subsection{Breuil modules with descent data}
\label{sec:breuil-modules-with}

Let $k$ be a
finite extension of $\F_p$, define $K_0=W(k)[1/p]$, and let $K$ be a
finite totally ramified extension of $K_0$ of degree~$e'$.  Suppose
that $L$ is a subfield of $K$ containing $\Qp$ such
that $K/L$ is Galois and tamely ramified.  Assume further that there is a uniformiser $\pi$
of $\bigO_{K}$ such that $\pi^{e(K/L)}\in L$, where $e(K/L)$ is the
ramification
degree of $K/L$, and fix such a $\pi$.
Since $K/L$ is tamely ramified, the category of Breuil modules with
coefficients and descent data is easy to describe (see \cite{sav06}). Let $k_E$ be a finite extension of $\Fp$.
The category $\BrMod_{\dd,L}$
consists of quadruples $(\mathcal{M},\Fil^1 \mathcal{M},\phi_{1},\{\widehat{g}\})$ where:

\begin{itemize}\item $\mathcal{M}$ is a finitely generated
  $(k\otimes_{\F_p} k_E)[u]/u^{e'p}$-module, free over $k[u]/u^{e'p}$.
\item $\Fil^1 \M$ is a $(k\otimes_{\F_p} k_E)[u]/u^{e'p}$-submodule of $\M$ containing $u^{e'}\M$.
\item $\phi_{1}:\Fil^1\M\to\M$ is $k_E$-linear and $\phi$-semilinear
  (where $\phi:k[u]/u^{e'p}\to k[u]/u^{e'p}$ is the $p$-th power map)
  with image generating $\M$ as a $(k\otimes_{\F_p} k_E)[u]/u^{e'p}$-module.
\item $\widehat{g}:\M\to\M$ for each
  $g\in\Gal(K/L)$  are additive bijections that preserve $\Fil^1 \M$, commute with the $\phi_1$-,
  and $k_E$-actions, and satisfy $\widehat{g}_1\circ
  \widehat{g}_2=\widehat{g_1\circ g}_2$ for all
  $g_1,g_2\in\Gal(K/L)$. Furthermore
  $\widehat{1}$ is the identity, and  if $a\in k\otimes_{\F_{p}} k_E$, $m\in\M$ then $\widehat{g}(au^{i}m)=g(a)((g(\pi)/\pi)^{i}\otimes 1)u^{i}\widehat{g}(m)$.\end{itemize}

The category $\BrMod_{\dd,L}$ is equivalent to the category of finite flat
group schemes over $\mathcal{O}_K$ together with a $k_E$-action and descent
data on the generic fibre from $K$ to $L$ (this equivalence depends on $\pi$).

We choose in this paper to adopt the conventions of \cite{bm} and
\cite{sav04}, rather than those of \cite{bcdt}; thus rather than working
with the usual contravariant equivalence of categories, we work with a
covariant version of it, so that our formulae for generic fibres will differ
by duality and a twist from those following the conventions of \cite{bcdt}.
To be precise, we obtain the associated $G_{L}$-representation (which we will refer to as the generic fibre) of an object of $\BrMod_{\dd,L}$ via
the functor $T_{\st,2}^{L}$.

Let $E$ be a finite extension of $\Qp$ with integers $\OO_E$, maximal
ideal $\m_E$, and residue field $k_E$.  Recall
from \cite[Sec. 2]{sav04} that the functor $D_{\st,2}^{K}$ is an equivalence of categories
  between the category of $E$-representations of $G_L$ which are semistable when
  restricted to $G_K$ and have Hodge-Tate weights in $\{0,1\}$, and
  the category of weakly admissible filtered $(\phi,N)$-modules $D$ with
  descent data and $E$-coefficients such that $\Fil^0 (K \otimes_{K_0}
  D) = K \otimes_{K_0} D$ and $\Fil^2 (K \otimes_{K_0}
  D) = 0$.

Suppose that $\rho$ is a representation
  in the source of $D_{\st,2}^{K}$.  Write $S = S_{K,\OO_E}$ (notation
  and terminology
in this paragraph are as in
  \cite[Sec. 4]{sav04}).  Then $T_{\st,2}^{L}$ is an
  essentially surjective functor from strongly divisible modules $\M$ (with
  $\OO_E$-coefficients and descent data) in $S[1/p]
  \otimes_{K_0\otimes E}
  D_{\st,2}^{K}(\rho)$ to Galois-stable $\OO_E$-lattices in~$\rho$.  This
  functor is compatible with reduction mod $\m_E$, so that
applying $T_{\st,2}^{L}$ to the object $(k \otimes_{\Fp}
k_E)[u]/(u^{e'p}) \otimes_{S/\m_E S} (\M/\m_E \M)$ of $\BrMod_{\dd,L}$ yields a
reduction mod $p$ of $\rho$ (see \cite[Cor. 4.12, Prop 4.13]{sav04}).

To simplify notation, for the remainder of the paper we write simply
$\M/\m_E\M$ for the above reduction mod ~$\m_E$ of $\M$ in
$\BrMod_{\dd,L}$ (we will never mean the literal $S/\m_E S$-module).
Let $\ell$ be the residue field of $L$, and let $\ur_{\lambda}$ 
denote the unramified character of $G_L$ sending an
arithmetic Frobenius element to $\lambda$.  Define $\Nm_{\ell/\F_p,k_E} :
(\ell \otimes_{\Fp} k_E)^{\times} \rightarrow (\Fp \otimes_{\Fp} k_E)^{\times} \cong k_E^{\times}$ to
be
the norm map $x \mapsto \prod_{\beta \in \Gal(\ell/\Fp)} \beta(x)$,
with each $\beta$ acting trivially on $k_E$. 

The following lemma is a
 more precise version of  \cite[Lem. 5.2]{geesavitttotallyramified}.

\begin{lem}
  \label{lem:connected_breuil_module}
  Let $\overline{\chi} : \Gal(K/L) \rightarrow k_E^{\times}$ be a
  character, and for $c \in (\ell \otimes_{\Fp} k_E)^{\times}$  let
  $\M(\overline{\chi},c)$ denote the Breuil
  module with $k_E$-coefficients and descent data from $K$ to $L$ that
  is free of rank one with
generator $v$ and
$$ \Fil^1 \M(\overline{\chi},c) = \M(\overline{\chi},c),\qquad \phi_1(v) =
cv, \qquad \widehat{g}(v) = (1 \otimes \overline{\chi}(g))v$$
for $g \in \Gal(K/L)$.  Then $T_{\st,2}^{L}(\M(\overline{\chi},c)) =
\ur_{\lambda} \cdot
\overline{\chi}$, where $\lambda = \Nm_{\ell/\Fp,k_E}(c)^{-1}$.
\end{lem}

\begin{proof} This statement is exactly the same as
  \cite[Lem. 5.2]{geesavitttotallyramified}, except that here we determine the
  unramified character multiplying $\overline{\chi}$.  We return ourselves to the proof of that
  statement, and in particular we re-adopt the notation from that proof, so that $\chi$
  is the Teichm\"uller lift of $\overline{\chi}$, the element $\widetilde{c} \in (W(\ell) \otimes_{\Zp}
  \OO_E)^{\times}$ is a lift of $c$, and $D:=D(\chi,\widetilde{c})$ is a
filtered $\phi$-module of rank one over $K_0 \otimes_{\Qp} E$ with descent
data $\chi$ and generator
$\mathbf{v}$
such that $\phi(\mathbf{v}) = p \widetilde{c} \mathbf{v}$; moreover the
filtration on $D_K = K \otimes_{K_0} D$ vanishes in degree $2$.  It is
possible to choose
$\widetilde{c}$ to be an element of finite multiplicative
order, and we do so.  

The representation $V_{\st,2}^{L}(D)$ giving rise to $D$ is equal to
$(B_{\st} \otimes_{K_0} D)^{\phi=p}_{N=0} \cap \Fil^1(B_{\mathrm{dR}}
\otimes_K D_K)$ (see the definition after
\cite[Cor. 2.10]{sav04}), and so is generated by some $\alpha \mathbf{v}$
with $\alpha \in (\Fil^0 B_{\mathrm{cris}}) \otimes_{\Qp} E$; then $p\alpha \mathbf{v} = \phi(\alpha
\mathbf{v}) = \phi(\alpha) \widetilde{c} p \mathbf{v}$, so that
$\phi(\alpha) \widetilde{c} = \alpha$.  If $f = [\ell : \Fp]$ it
follows that $\phi^{(f)}(\alpha) = \Nm_{L_0/\Qp,E}(\widetilde{c})^{-1}
\alpha$, where $L_0 = W(\ell)[1/p]$ and $\Nm_{L_0/\Qp,E}$ is defined
via the obvious analogy with $\Nm_{\ell/\Fp,k_E}$.   Set
$\widetilde{\lambda} = \Nm_{L_0/\Qp,E}(\widetilde{c})^{-1}$.

Since $\widetilde{c}$ has
finite order we have $\phi^{(m)}(\alpha) = \alpha$ for some $m > 0$,
and therefore $\alpha$ is an element of $(\Fil^0 B_{\mathrm{cris}})^{\phi^{(m)} = 1}
\otimes_{\Qp} E = \Q_{p^m} \otimes_{\Qp} E$.  In particular,
the action of crystalline Frobenius coincides with the
action of an arithmetic Frobenius on $\alpha$.
As a result, if $g \in G_L$ is a lift of an $n$th power of arithmetic
Frobenius with $n \in \Z$ then $g(\alpha \mathbf{v}) =
\phi^{(nf)}(\alpha) \chi(g) \mathbf{v} = (\ur_{\widetilde{\lambda}}
\cdot \chi)(g) \alpha \mathbf{v}$.
Since $\widetilde{\lambda}$ lifts $\lambda$, the result follows by continuity.
\end{proof}

For the remainder of this paper we make the hypothesis that $k_E$ is
sufficiently large as to contain an embedding of $k$.  Let $\sigma_0$
be a fixed choice of embedding $k \hookrightarrow k_E$ and recursively
define $\sigma_{i+1}^p = \sigma_i$.  If $M$ is any $(k \otimes_{\Fp} k_E)$-module,
 we recall from \cite{sav06} that $M$ decomposes as 
a direct sum $M = \oplus_{i=0}^{d-1} M_{i}$, where $d=[k:\Fp]$ and
$M_{i}$ is the $k_E$-submodule on which multiplication by
$x\otimes 1$ for $x \in k$ is the same as multiplication by $1
\otimes \sigma_i(x)$.  In fact there is a collection of idempotents
$e_i \in k \otimes_{\Fp} k_E$ so that $M_{i} = e_i M$ and 
$\phi(e_i) = e_{i+1}$.  

Suppose now that $\M$ is an object of $\BrMod_{\dd,L}$.  Note that
$\phi_1$ maps $(\Fil^1 \M)_{i}$ into $\M_{i+1}$.  For $g
\in G_L$ let
$\overline{\eta}(g)$ be the image of $g(\pi)/\pi$ in (the $e(K/L)$th
roots of unity of) $k$.   The rank
one objects of $\BrMod_{\dd,L}$ are classified as follows.

\begin{prop}[\cite{sav06}, Theorem 3.5] \label{prop:rank one breuil modules} With our fixed choice of uniformiser
 $\pi$, every rank one object of $\BrMod_{\dd,L}$ with descent data
 relative to $L$ has the form:
\begin{itemize}
\item $\M = ((k \otimes_{\Fp} k_E)[u]/u^{e'p}) \cdot v $,
\item $(\Fil^1 \M)_{i} = u^{r_i} \M_{i}$,
\item $\phi_1(\sum_{i=0}^{d-1} u^{r_i} e_i v) = cv$ for some $c \in (\ell
  \otimes_{\Fp} k_E)^{\times}$, and
\item $\widehat{g}(v) = \sum_{i=0}^{d-1} (\overline{\eta}(g)^{k_i}
  \otimes 1) e_iv$ for all $g \in \Gal(K/L)$,
\end{itemize} 
where $0 \le r_i \le e'$ and $0 \le k_i < e(K/L)$ are sequences of
integers satisfying $k_i \equiv p(k_{i-1} + r_{i-1}) \pmod{e(K/L)}$;
furthermore the sequences $r_i, k_i$ are periodic with period dividing
$f = [\ell : \Fp]$.
\end{prop}

\begin{cor}\label{cor:chars for rank one breuil modules}
     In the above proposition, suppose that $e(K/L)$ is 
divisible by $p^f-1$.  Define $s_0 = p(p^{f-1} r_0 + \cdots +
r_{f-1})/(p^f-1)$ and $\lambda = \Nm_{\ell/\Fp,k_E}(c)^{-1}$.
Then $T_{\st,2}^{L}(\M) = (\sigma_0 \circ
\overline{\eta}^{k_0 + s_0}) \cdot \ur_{\lambda}$.
\end{cor}

\begin{rem}
  \label{rem:well-defined-rank-one}
According to 
  \cite[Rem. 3.6]{sav06}, the congruences $k_i \equiv p(k_{i-1} + r_{i-1})
  \pmod{e(K/L)}$ imply that 
$$p^{f-1} r_0 + \cdots + r_{f-1} \equiv 0 \pmod{p^f-1}, $$
and $k_0$ is a solution
to $-p(p^{f-1} r_0  + \cdots + r_{f-1}) \equiv (p^{f}-1)k_0 \pmod{e(K/L)}$. 
 It follows that $s_0$ is an integer; moreover $(p^f-1)(k_0 + s_0)
  \equiv 0 \pmod{e(K/L)}$, so that the image of $\overline{\eta}^{k_0 + s_0}$
  lies in $\ell^{\times}$ and $\overline{\eta}^{k_0+s_0}$ is actually
  a character.
\end{rem}

\begin{proof}[Proof of Corollary~\ref{cor:chars for rank one breuil modules}]   We proceed as in Example 3.7 of \cite{sav06}.

Define $s_i = p(r_i p^{f-1} + \cdots + r_{i+f-1})/(p^f-1)$ with
subscripts
taken modulo $f$, and observe that $(k_i + s_i) \equiv p^i(k_0 + s_0)
\pmod{e(K/L)}$.  Let $\overline{\chi} = \sigma_0 \circ
\overline{\eta}^{k_0 + s_0}$.
We check that there is a morphism $\M(\overline{\chi},c) \rightarrow
\M$ with $\M(\overline{\chi},c)$ as in Lemma 
  \ref{lem:connected_breuil_module} (except that here we will use $w$
  to denote its generator, since $v$ is now our generator of~$\M$).

The morphism will send $w \mapsto \sum_i u^{s_i} e_i v$.  One checks
easily that this is a morphism of Breuil modules.  Indeed: the
filtration is preserved since $s_i \ge r_i$; the morphism commutes
with $\phi_1$ because $s_{i+1} = p(s_i - r_i)$; and
to check that the morphism commutes with descent data,
use the fact that $\widehat{g}(w) = (1 \otimes (\sigma_0 \circ
\overline{\eta}^{k_0 + s_0}(g))) w = \sum_i (\overline{\eta}^{k_i + s_i}(g)
\otimes 1) e_i w$.

Now the claim follows immediately from  Lemma 
  \ref{lem:connected_breuil_module}  and an application of
  \cite[Prop. 8.3]{SavittCompositio}.  (This last step uses our running hypothesis
  that $p > 2$.)
\end{proof}

\subsection{Necessary conditions: notation and preliminaries}
\label{sec:necess-cond-notat}

Let $\p$ be a prime of $F$ lying above $p$, and $\pi_{\p} \in \p$ our chosen
uniformiser.  Suppose that the residue field of $F_{\p}$ has order
$q=p^f$. 

In the remainder of Section~\ref{sec: necessity}, we consider the following situation.  Let $L$ be the unramified
quadratic extension of $F_{\p}$, and $K$ the splitting field of
$u^{p^{2f}-1}-\pi_{\p}$ over $L$.  Let $\varpi$ be a choice of
$\pi_{\p}^{1/(p^{2f}-1)}$ in $K$.  Let $k$ denote the
residue field of $K$, and if $g\in\Gal(K/F_{\p})$ then as before we let $\overline{\eta}(g)$
be the image of $g(\varpi)/\varpi$ in $k$. Suppose that $F_{\p}$ has
absolute ramification index $e$, and write $e'=e(p^{2f}-1)$.  (We
alert the reader
that in what follows, the fields $F_{\p}$ and $L$ will both take
turns being used
in the role of the field $L$ of the previous subsection.)

Suppose that $k$ embeds into $k_E$.  By Proposition \ref{prop:rank one breuil modules}, any rank one Breuil module
$\M$ with $k_E$-coefficients and descent
data from $K$ to $L$
may be written in the form
\begin{itemize}
\item $\M=((k \otimes_{\Fp} k_E)[u]/u^{e(p^{2f}-1)p}) \cdot v$, 
\item $(\Fil^1 \M)_i =u^{r_i}\M_i$,
\item $\phi_1(u^{r_i}e_iv)=(1 \otimes \gamma_i)e_{i+1}v$ for some $\gamma_i\in k_E^\times$,
\item $\widehat{g}(\sum_{i=0}^{2f-1} e_iv)=\sum_{i=0}^{2f-1}
  (\overline{\eta}(g)^{k_i} \otimes 1)e_iv$ for all $g\in\Gal(K/L)$.
\end{itemize}
Here the $k_i$, $r_i$ are any integers with $k_i\in[0,p^{2f}-1)$ and
$r_i\in[0,e(p^{2f}-1)]$ satisfying $k_{i+1}\equiv
p(k_i+r_i)\pmod{p^{2f}-1}$.  For $g \in \Gal(K/L)$ we write $\widehat{g}(e_i v) = (1
\otimes \chi_i(g)) e_i v$ where $\chi_i = \sigma_i \circ
\overline{\eta}^{k_i}$.  Note that
$\chi_i$, defined on $\Gal(K/L)$, is a homomorphism. 

Let $\chi : I_{F_{\p}} \to \OO_E^{\times}$ be an inertial character
with $\chi = \chi^{q^2}$ but $\chi \neq \chi^q$, and let
$\overline{\chi}$ denote its reduction modulo the maximal ideal of
$\OO_E$.  In what follows we will be concerned with Breuil
modules $\M$ as above that have the extra property $\chi_i \in
\{\overline{\chi},\overline{\chi}^q\}$ for all $i$.  In the remainder
of this subsection we introduce some notation that is special to this
situation (and that will be used repeatedly throughout the rest of the paper), and we
derive a variant of Corollary~\ref{cor:chars for rank one breuil
  modules}.

Let $\eta_i = (\sigma_i \circ \overline{\eta})|_{I_{F_{\p}}}$ for $0 \le i < 2f$ be a system of
fundamental characters of niveau $2f$ of $I_{F_{\p}}$; note that $\eta_i^p
= \eta_{i-1}$.   Then 
$\omega_i  = \eta_i \eta_{i+f}$ for $0 \le i < f$ is a system of fundamental characters of niveau $f$.
Write 
\begin{equation}\label{eq:chiproduct}
\overline{\chi}=\prod_{i=0}^{2f-1}\eta_i^{c_i}
\end{equation} with
$0\leq c_i\leq p-1$; since $\overline{\chi}$ is nontrivial, this is
unambiguous.  We let $J$ be the set of $i \in \{0,\ldots,2f-1\}$ such that  $\chi_i
= \overline{\chi}$.  References to elements of $J$ should always be
taken modulo $2f$, so that e.g. if $i = 2f-1$ then $i+1$ refers to $0$.

 The congruence $k_{i+1}\equiv
p(k_i+r_i)\pmod{p^{2f}-1}$ is equivalent to the relation $\chi_{i+1} =
\chi_i \eta_i^{r_i}$.  If $i\in J$ and $i+1\in J$, or if $i\notin J$ and
$i+1\notin J$, this gives $r_i\equiv 0\pmod{p^{2f}-1}$. In
either case, write $r_i=(p^{2f}-1)x_i$ for some $0\leq x_i\leq e$. If
$i\in J$ and $i+1\notin J$, we see that \[r_i=(p^{2f}-1)x_i+
(p^f-1)(p^{f-1}(c_{i+f+1}-c_{i+1})+p^{f-2}(c_{i+f+2}-c_{i+2})+\dots+(c_i-c_{i+f}))\] for some $x_i$, and  if $i\notin J$ and $i+1\in J$, then \[r_i=(p^{2f}-1)x_i+
(p^f-1)(p^{f-1}(c_{i+1}-c_{i+f+1})+p^{f-2}(c_{i+2}-c_{i+f+2})+\dots+(c_{i+f}-c_{i})).\]

Since the
expression
$(p^f-1)(p^{f-1}(c_{i+f+1}-c_{i+1})+p^{f-2}(c_{i+f+2}-c_{i+2})+\dots+(c_i-c_{i+f}))$
is nonzero and is strictly between $1-p^{2f}$ and $p^{2f}-1$, we
allow either $0\leq x_i\leq e-1$ or $1\leq x_i\leq e$, depending on whether
the sign of this expression is positive or negative.  If $i\in J$ and $i+1\notin J$, then the
allowable range is $0\leq x_i\leq e-1$ precisely when there is a $j\geq
1$ with $c_{i+k}=c_{i+k+f}$ for all $1\leq k<j$ and
$c_{i+j+f}>c_{i+j}$, and the situation is reversed in the case
$i\notin J$ and $i+1\in J$.  We summarize these conditions in the
following definition.

\begin{defn}
  \label{defn:allowable}
  Fix $J$ and $\overline{\chi}$ as above.  We say that $x_i \in \{0,1,\ldots,e\}$ is
  \emph{allowable} in each of the following situations, and \emph{not
    allowable} otherwise.
\begin{itemize}
\item $i,i+1 \in J$ or $i,i+1 \not\in J$;

\item $i\in J$, $i+1 \not\in J$: we require $x_i \neq e$ if there is $j\geq
1$ with $c_{i+k}=c_{i+k+f}$ for all $1\leq k<j$ and
$c_{i+j} < c_{i+j+f}$; we require $x_i \neq 0$ otherwise;

\item $i\not\in J$, $i+1 \in J$: we require $x_i \neq 0$ if there is $j\geq
1$ with $c_{i+k}=c_{i+k+f}$ for all $1\leq k<j$ and
$c_{i+j} < c_{i+j+f}$; we require $x_i \neq e$ otherwise.
\end{itemize}
Here subscripts should be taken modulo $2f$.  We also say $x_i$ is not
allowable if $x_i \not\in \{0,1,\ldots,e\}$.  We say that the list
$x_0,\ldots,x_{2f-1}$ is allowable if each $x_i$ is allowable.
\end{defn}

Thus a rank one Breuil module $\M$ with the property that
$\chi_i \in \{\overline{\chi},\overline{\chi}^q\}$ for all $i$ gives
rise to a set $J$ and an allowable collection $x_0,\ldots,x_{2f-1}$.
Conversely, it is straightforward to check that this construction can
be reversed: any $J$ and any
allowable list $x_0,\ldots,x_{2f-1}$, together with any choice of
$\gamma_i$'s, determines a Breuil module $\M$
with the desired property.  

Let $\psi$ denote the restriction to inertia of $T_{\st,2}^L(\M)$, and
note from Corollary~\ref{cor:chars for rank one breuil modules} that
$\psi$ depends only on the $\chi_i$'s and $r_i$'s, or equivalently only
on $J$ and the $x_i$'s.

\begin{lem}
  \label{lem:computationofpsi}
  Let $\M$ be a rank one Breuil module with $k_E$-coefficients and descent data
  from $K$ to $L$ with $\chi_i \in
  \{\overline{\chi},\overline{\chi}^q\}$ for all $i$.  Then 
 \begin{equation}
\label{eq:psiformula}
\psi =\prod_{i\in
  J}\eta_{i}^{c_i}\prod_{i\notin J}\eta_i^{c_{i+f}}\prod_{i=0}^{2f-1}\eta_i^{x_i}.
\end{equation}
\end{lem}

\begin{proof}
Recall that \[s_0=\frac{p}{(p^{2f}-1)}(r_0p^{2f-1}+r_1p^{2f-2}+\dots+r_{2f-1}),\]
so that $\psi =\eta_0^{k_0+s_0} = \chi_0 \eta_0^{s_0}$. Write
$s_0$ as $p^{2f} x_0 + p^{2f-1} x_1 + \cdots + p x_{2f-1}$ plus a
linear expression in the $c_i$'s.  

We compute the coefficient of $c_0$ in this linear expression.
For each transition  $i
\in J$, $i+1 \not\in J$ with $i \in [0,f)$, the coefficient of $c_0$
in $p^{2f-i} r_i$ is $p^{2f}(p^f-1)$; on the other hand for each
transition $i \not\in J$, $i+1 \in J$ with $i \in [0,f)$ the coefficient
of $c_0$ in $p^{2f-i} r_i$ is
$-p^{2f}(p^f-1)$.  

For $i \in [f,2f-1)$ the respective coefficients
are $-p^{f}(p^f-1)$ for transitions $i \in J$, $i+1 \not\in J$, and
$p^f(p^f-1)$ for the reverse.  As a consequence:
\begin{itemize}
\item  If $0,f \in J$ or $0,f \not\in J$ then the net
number of transitions out of $J$ from $i=0$ to $i=f$ is zero, and similarly
from $f$ to $2f$.  In either case the coefficient of $c_0$ in $s_0$ is zero.

\item If $0 \in J$ and $f \not\in J$, then the net number of transitions
out of $J$ from $i=0$ to $i=f$ is $1$, and from $i=f$ to $i=2f$ is
$-1$.  In this case the coefficient of $c_0$ in $s_0$ is
$(p^{2f}(p^f-1) + p^f(p^f-1))/(p^{2f}-1) = p^f$.

\item Similarly if $0 \not\in J$ and $f\in J$, the coefficient of
  $c_0$ in $s_0$ is $-p^f$.
\end{itemize}

From (\ref{eq:chiproduct}) and the definition of $J$, the contribution of $c_0$ to $\chi_0$ is $\eta_0^{c_0}$ if $0 \in
J$ and $\eta_0^{p^f c_0}$ if $0 \not\in J$.  Thus the total
contribution of $c_0$ to $\psi = \chi_0 \eta_0^{s_0}$ is:
\begin{itemize}
\item $\eta_0^{c_0}$ if $0 \in J$, $f \in J$;
\item $\eta_0^{c_0}\eta_0^{p^f c_0} = \eta_0^{c_0} \eta_f^{c_0}$ if $0
  \in J$, $f \not\in J$;
\item $\eta_0^{-p^f c_0} \eta_0^{p^f c_0} = 1$ if $0 \not\in J$, $f
  \in J$;
\item $\eta_0^{p^f c_0} = \eta_f^{c_0}$ if $0 \not\in J$, $f \not\in J$.
\end{itemize}
In each case we obtain precisely the contribution of $c_0$ to the
first two products on the right-hand side of (\ref{eq:psiformula}).  The lemma follows
by cyclic symmetry, together with the fact that $\eta_0$ raised to the
power $p^{2f} x_0 + \cdots + px_{2f-1}$ is the third product on the
right-hand side of (\ref{eq:psiformula}). 
\end{proof}

\subsection{Necessary conditions: the reducible case}\label{ss:reducible}

Suppose that $\rhobar:G_{F_{\p}}\to\GL_2(k_E)$ with $k_E$ a finite
field into which $k$ may be embedded, and assume that $\rhobar$ is
the reduction mod $\m_E$ of a parallel potentially Barsotti-Tate
representation $\rho$ of type $\chi \oplus \chi^q$.  Let $\mathcal{H}$ be the $\m_E$-torsion of the
Barsotti-Tate group over $\bigO_K$ corresponding to $\rho$; then
$\mathcal{H}$ is a finite flat group scheme over $\bigO_K$ with
descent data to $F_{\p}$, and $\rhobar$ is the generic
fibre of $\mathcal{H}$.

In this subsection we suppose that
$\rhobar\cong\begin{pmatrix}\psi_1&*\\0&\psi_2\end{pmatrix}$ is reducible, and we
wish to restrict the possibilities for
$\psi_1$ and $\psi_2$.  Note that by a standard scheme-theoretic
closure argument, $\psi_1$ corresponds to a finite flat subgroup scheme
$\mathcal{G}$ of $\mathcal{H}$.  Let $\M$ be the rank one Breuil
module with $k_E$-coefficients and
descent data from $K$ to $F_{\p}$ corresponding to $\mathcal{G}$, and
let $\chi_i$ for $i=0,\ldots,2f-1$ be defined as in the previous subsection.
It follows from Corollary \ref{cor:dd mod p},
which does not depend on anything in this paper before
Section~\ref{sec:types-and-descent}, that the
descent data for $\mathcal{H}$ is of the form
$\overline{\chi}\oplus\overline{\chi}^q$, so that
we have  $\chi_i \in \{\overline{\chi}, \overline{\chi}^q\}$ for all
$i$.  Therefore we may define $J$ and $x_0,\ldots,x_{2f-1}$ as in the
previous subsection, and the analysis of the previous subsection applies to $\M$.

Since the descent data on $\M$ is from $K$ to $F_{\p}$ and not
simply from $K$ to $L$, we in fact have from Proposition~\ref{prop:rank one
  breuil modules} that $r_{i+f} = r_i$ and $k_{i+f} = k_i$ for all
$i$, or equivalently $\chi_i = \chi_{i+f}^q$ and $x_i = x_{i+f}$ for
all $i$.  In particular for all $i$ we have exactly one of $i,i+f$ in  $J$, and $x_{i+f} = x_i$ is allowable if and only if $x_i$ is. 
Letting $\pi$ denote the natural projection from
$\Z/2f\Z$ to $\Z/f\Z$,  we deduce from
Lemma~\ref{lem:computationofpsi} that $\psi_1|_{I_{F_{\p}}}$
has the form
\begin{equation}\label{eq:niveau one formula}
\psi_1|_{I_{F_{\p}}}=\prod_{i\in
 J}\omega_{\pi(i)}^{c_i}\prod_{i=0}^{f-1}\omega_i^{x_i}
\end{equation}
where $J$ contains exactly one of $i,i+f$ for all $i$, and
$x_0,\ldots,x_{f-1}$ are allowable for~$\overline{\chi}$ and $J$.

\begin{prop}
  \label{prop:niveau-one-get-everything-if-e-large}
If $e\ge p-1$, then (for fixed $\overline{\chi}$) any
inertial character of niveau~$f$ occurs as the right-hand side of
\eqref{eq:niveau one formula} for some choice of $J$ with exactly one
of $i,i+f \in J$ for all $i$, and some allowable
values 
$x_0, \ldots, x_{f-1}$.
 \end{prop}

 \begin{proof}
The proposition is immediate if $e \ge p$, because for \emph{any} $J$ the
allowable range for each $x_i$ contains $p$ consecutive integers; so we suppose that
$e=p-1$, where the matter is more delicate.  Observe that the
claim is invariant under twisting $\overline{\chi}$ by a
character $\omega$ of niveau $f$: replacing $\overline{\chi}$ with
$\omega \overline{\chi}$ replaces each $\chi_i$ with $\omega \chi_i$,
leaving the possibilities for the integers $r_i$ and $s_0$ arising
from the relevant Breuil modules unchanged.
The claim is similarly invariant under replacing $\overline{\chi}$
with $\overline{\chi}^p$.  As a consequence of these observations we may suppose without loss of
generality that $c_{0},\ldots,c_{f-1}=0$ while $c_{2f-1} \neq 0$.

Consider first the set $J=\{0,\ldots,f-1\}$.   The allowable range for
$x_{f-1}$ is $[1,p-1]$ (since there is some $1 \le j \le f$
with $c_{f-1+j} > 0$ while each $c_{j-1}=0$), and $x_0,\ldots,x_{f-2}$ can
range over $[0,p-1]$.  Writing the right-hand side of \eqref{eq:niveau
  one formula} as $\omega_{f-1}$ raised to the power $p^{f-1} x_0 + \cdots
+ px_{f-2} + x_{f-1}$, we see that the exponent of $\omega_{f-1}$ obtains every
integer value in $[0,p^{f}-1]$ except those divisible by $p$.

Now consider the sets $J = \{2f-i,\ldots,2f-1,0,\ldots,f-i-1\}$ with
$1 \le i \le f-1$.  The allowable range for $x_{f-i-1}$ is $[0,p-2]$
since $c_{2f-1} > 0$ while each $c_{f-i},\ldots,c_{f-1}$ equals $0$;
for each other $x_i$ the allowable range is $[0,p-1]$.  For this
choice of $J$ the
right-hand side of \eqref{eq:niveau one formula} becomes $\omega_{f-1}$
raised to the power 
\begin{equation}\label{eq:basepsum}
(p^{i-1} c_{2f-i} + \cdots +  c_{2f-1})  + (p^{f-1} x_0 + \cdots
+ px_{f-2} + x_{f-1}).\end{equation} The right-hand term varies over all integers in the
range $[0,p^f-1]$ except those whose $p^i$-coefficient in base $p$ is
$p-1$.  In particular the base $p$ sum in \eqref{eq:basepsum} does not
have a carry from the $p^i$-place to the $p^{i+1}$-place.   Since
$c_{2f-1} \neq 0$, it follows that the values taken by
\eqref{eq:basepsum} (with allowable $x_0,\ldots,x_{f-1})$
include all integers in $[0,p^{f}-1]$ that are exactly divisible by
$p^i$.

All together, we find that for suitable choices of $J$ the right-hand
side of \eqref{eq:niveau one formula} when written as a power of
$\omega_{f-1}$ can take every exponent in the range $[1,p^{f}-1]$.
This is a complete set of powers of $\omega_{f-1}$.
\end{proof}

\subsection{Necessary conditions: the irreducible case}\label{ss:irreducible}

We retain the notation and hypotheses of the previous subsection, but
now we consider the case of an irreducible~$\rhobar$. In this case,
$\rhobar|_{G_L}\equiv\begin{pmatrix}\psi_1&0\\0&\psi_2\end{pmatrix}$
with $\psi_2 = \psi_1^q$.
Again, we examine the possibilities for $\psi_1|_{I_{F_{\p}}}$. Let
$\mathcal{H}$ be the finite flat group scheme with generic fibre descent data from
$K$ to $L$ corresponding to
$\rhobar |_{G_{L}}$, and let $\mathcal{G}$ be the finite flat subgroup scheme corresponding
to~$\psi_1$.  Note that the descent data on $\mathcal{H}$ must
extend to $\Gal(K/F_{\p})$ while the descent data on $\mathcal{G}$
must not.

Let $\M$ be the Breuil module with $k_E$-coefficients and descent
data from $K$ to $L$ corresponding
to $\G$.   It follows once again from Corollary \ref{cor:dd mod p} that the
descent data for $\mathcal{H}$ is of the form
$\overline{\chi}\oplus\overline{\chi}^q$, so that
we have  $\chi_i \in \{\overline{\chi}, \overline{\chi}^q\}$ for all
$i$.  Thus the analysis of subsection~\ref{sec:necess-cond-notat}
applies to $\M$, and we may employ the notation of that subsection; in
particular we take $J = \{ i : \chi_i  = \overline{\chi} \}$.

\begin{rem}\label{rem:oops}  In an earlier version of this paper we claimed to show that
  $\mathcal{H}$ must decompose as a product $\mathcal{G} \times
  \mathcal{G}'$ where $\mathcal{G}'$ is the finite flat subgroup
  scheme with descent data from $K$ to $L$ corresponding to $\psi_2$, in which case it
  would follow that $i \in J$ if and only if $i+f \in J$, and also that
  $x_i + x_{i+f} = e$ for all $i$.  Our proof of this claim was in error and we do
  not know whether or not the claim is true.  Nevertheless we have the
  following proposition.
\end{rem}

\begin{prop}
  \label{prop:irreduciblecase}
  There exists $J \subset \{0,\ldots,2f-1\}$ with $i \in J$ if and
  only if $i+f \in J$, and an allowable list $x_0,\ldots,x_{2f-1}$
  with $x_i + x_{i+f} = e$ for all $i$, such that
\begin{equation}
\label{eq:irredpsi}
\psi_1 |_{I_L} = \prod_{i\in
  J}\eta_{i}^{c_i}\prod_{i\notin
  J}\eta_i^{c_{i+f}}\prod_i\eta_i^{x_i} .
\end{equation}
\end{prop}

Per Remark~\ref{rem:oops}, the $J$ in the proposition may not
necessarily be the $J$ coming from~$\M$.  The proof of Proposition~\ref{prop:irreduciblecase} occupies the remainder of this section.

\begin{lem}\label{lem:determinant of modular of some weight}If
  $\rhobar$ has a parallel potentially Barsotti-Tate lift of type
$\chi \oplus\chi^{q}$, then $\det
\rhobar|_{I_{F_\p}}=\epsilon \cdot \overline{\chi}^{q+1}$, where $\epsilon$ is the
mod $p$ cyclotomic character.
\end{lem}
\begin{proof} This follows at once from Definition \ref{defn:barsotti tate lifts} and the results of section B.2 of \cite{cdt}.
\end{proof}

 By the Lemma, we must have
$(\psi_1|_{I_L})^{q+1} = \epsilon \cdot \overline{\chi}^{q+1}$.   A
straightforward computation shows that for any $J$ with $i \in J$ if
and only if $i+f \in J$, and any $x_0,\ldots,x_{2f-1}$ with
$x_i + x_{i+f} = e$ for all $i$, the 
character $\psi$ on the right-hand side of (\ref{eq:irredpsi}) has
$\psi^{q+1} = \epsilon \cdot \overline{\chi}^{q+1}$ as well; this uses
the fact that $\epsilon = \omega_0^e \cdots \omega_{f-1}^e$.

\begin{lem}
  \label{lem:get-everything-if-e-is-large}
  If $e \ge p-1$, then as $J$ varies over subsets of
  $\{0,\ldots,2f-1\}$ with $i \in J$ if and only if $i+f \in J$, and
  $x_0,\ldots,x_{2f-1}$ varies over allowable lists for $J$ with $x_i+ x_{i+f}
  = e$ for all $i$, the right-hand side of (\ref{eq:irredpsi}) varies over all
  inertial characters $\psi$ with $\psi^{q+1} = \epsilon \cdot
  \overline{\chi}^{q+1}$.  In particular
  Proposition~\ref{prop:irreduciblecase} is true if $e \ge p-1$.
\end{lem}

\begin{proof}
First take $J=\{0,\ldots,2f-1\}$, so that any $x_0,\ldots,x_{2f-1} \in
[0,e]$  are allowable.  We let $x_0,\ldots,x_{f-1}$ vary over $[0,e]$
and take $x_{i+f} = e-x_i$, and we 
consider the characters $\psi$ that occur.

 If $X = p^{f-1} x_{0} + \cdots + x_{f-1}$ then $\psi
 = \overline{\chi} \cdot \eta_{2f-1}^{e(p^f-1)/(p-1)}\cdot
\eta_{f-1}^{(1-p^f) X}$ and depends only on $X \pmod{p^f+1}$.  If $e \ge p$ then as $x_0,\ldots,x_{f-1}$ range over
the interval $[0,e]$, the integer $X$ ranges over an interval that
includes $[0,p^f]$, and $\psi$ ranges over all $p^f+1$ inertial
characters $\psi$ with $\psi^{q+1} = \epsilon \cdot \overline{\chi}^{q+1}$.

If instead $e=p-1$, then $X$ only ranges over the interval
$[0,p^f-1]$, and we obtain all possibilities for $\psi$
satisfying the condition on $\psi^{q+1}$ except $\psi = \overline{\chi}$.  
But performing the same analysis with $J = \emptyset$ gives us all
possibilities for $\psi$ except $\psi = \overline{\chi}^q$; in particular
since $\overline{\chi} \neq \overline{\chi}^q$
we obtain $\psi = \overline{\chi}$ as a possibility with $J
= \emptyset$.
  \end{proof}

Before continuing with the proof of Proposition~\ref{prop:irreduciblecase}, we make the following
observation.  Suppose that $i \in J$ if and only if $i+f \in J$, and
$x_0,\ldots,x_{2f-1}$ is an allowable list such that the product $\psi$
on the right-hand side  of (\ref{eq:irredpsi}) satisfies $\psi^{q+1} =
\epsilon \overline{\chi}^{q+1}$.  If $e < p-1$, then the condition
$x_i+x_{i+f} = e$ must be satisfied automatically.  Indeed, the
condition $\psi^{q+1} =\epsilon \overline{\chi}^{q+1}$ comes down to
$\prod_{i=0}^{f-1} \omega_i^{x_i + x_{i+f}} = \prod_{i=0}^{f-1}
\omega_i^e$.  Since $x_i+x_{i+f} \in [0,2e]$ and $e < p-1$, the only
possibility is $x_i + x_{i+f} = e$ for all $i$.  

  \begin{proof}[Proof of Proposition~\ref{prop:irreduciblecase}]
    Thanks to Lemma~\ref{lem:get-everything-if-e-is-large} we may
    assume $e < p-1$.   Let $J$ be any subset of $\{0,\ldots,2f-1\}$,
    let $x_0,\ldots,x_{2f-1}$ be allowable for $J$,
    and write 
$$ \psi  = \prod_{i \in J} \eta_i^{c_i} \prod_{i \not\in J}
\eta_i^{c_{i+f}} \prod_i \eta_i^{x_i}.$$
We wish to prove that if $\psi^{q+1} = \epsilon \cdot
\overline{\chi}^{q+1}$ (so that, for instance, $J$ and
$x_0,\ldots,x_{2f-1}$ might be the data associated to $\M$) then there exists some $J'$ with $i \in J'$ if
and only if $i+f \in J'$, and allowable $x'_0,\ldots,x'_{2f-1}$ such
that if we write
$$ \psi'  = \prod_{i \in J'} \eta_i^{c_i} \prod_{i \not\in J'}
\eta_i^{c_{i+f}} \prod_i \eta_i^{x'_i} $$ then in fact we have $\psi =
\psi'$.  (Then the desideratum $x'_i +
x'_{i+f} = e$ also holds, by the observation immediately before we
began the remainder of the proof.)

Assuming that $J$ does not already satisfy $i \in J$ if and only if $i+f \in J$,
it suffices to produce $J'$ and allowable $x'_0,\ldots,x'_{2f-1}$ such
that $\psi'=\psi$ and $J'$ has more pairs $(i,i+f)$ with $i \in J'$ if and only if $i+f \in J'$
than $J$ does.  (Then repetition of this step will complete the argument.)
This is what we now carry out.

Let $S$ be the set of indices $i$ such that $c_i$ appears as an
exponent twice in the product for $\psi$ (equivalently, such that $i \in J$ and $i+f
\not\in J$), and similarly let $T$
be the set of indices $i$ such
that $c_i$ occurs zero times (equivalently, $i \not\in J$ and $i+f \in
J$).  Note that $S = f + T$ (with the obvious
meaning for this notation).  Then the 
condition on $\psi^{q+1}$ is
$$ \prod_{i = 0}^{f-1} \omega_i^{x_i + x_{i+f}} \prod_{i \in T} \omega_i^{-c_i}
\prod_{i \in S} \omega_i^{c_i} = \prod_{i=0}^{f-1} \omega_i^{e} $$
which we re-write as
$$ \prod_{i=0}^{f-1} \omega_i^{x_i + x_{i+f} - e \pm [c_i - c_{i+f}]} = 1$$
where the brackets around $c_i - c_{i+f}$ denote that the term may not
occur (in this case, it occurs with sign $+$ if $i \in S$, with sign $-$ if $i
\in T$, and not at all if $i$ is in neither $S$ nor $T$). 

Each exponent in this product lies in the interval $[-e-p+1,e+p-1]
\subset [-(2p-3),(2p-3)]$ since  $e < p-1$.  Now, if $\prod_{i=0}^{f-1}
\omega_i^{y_i} = 1$ then the vector $(y_0,\ldots,y_{f-1})$ must be an
integral linear combination 
$$a_0 (p,0,\ldots,0,-1)  + a_1 (-1,p,0,\ldots,0) + \ldots + a_{f-1}
(0,\ldots,0,-1,p).$$
It is easy to
check that if each $y_i$ lies in $[-(2p-3),(2p-3)]$ then in fact each
$a_i$ must be $0$ or $\pm 1$.

Writing the vector $(x_i + x_{i+f} - e \pm [c_i - c_{i+f}])_i$ as such a
linear combination, we have
$$ x_i + x_{i+f} - e \pm [c_i - c_{i+f}] = pa_i - a_{i+1} $$ for all
$i \in \{0,\ldots,2f-1\}$ where we conventionally set $a_{i+f} :=
a_i$; take all subscripts modulo $2f$; and the sign is $+$ if $i \in
S$ and $-$ if $i \in T$, and $0$ otherwise.

Choose any \emph{maximal} interval $[j',j] \subset \Z$ such that $i \not\in J$ and $i+f \in J$ for $i
\in [j',j]$.  (As usual, we abuse
notation and take all indices modulo $2f$.)  By definition this
interval cannot contain both $i$ and $i+f$ for any $i$, so it contains at most $f$
integers.  Now if $i \in [j',j]$ we have $i \in T$ and
$i+f \in S$, so that in fact 
\begin{equation}\label{eq:keyequation}
x_i + x_{i+f} - e - (c_i - c_{i+f}) = pa_i - a_{i+1} \end{equation}
for $i \in [j',j]$.

First consider the case $j = j' + (f-1)$.   Define $J' =
\{0,\ldots,2f-1\}$, and set 
\begin{itemize}
\item  $x'_i = x_i - (c_i - c_{i+f}) - pa_i +
a_{i+1}$ if $i \in [j',j'+f-1)$,

\item $x'_{j'+f-1} = x_{j'+f-1} - (c_{j'+f-1} - c_{j'-1}) -
  pa_{j'-1}$,

\item $x'_{j'-1} = x_{j'-1} + a_{j'}$, and

\item $x'_i = x_i$ for all other indices $i$.
\end{itemize}    
One checks
easily that $\psi' = \psi$, and by construction $x_i'
+ x'_{i+f} = e$ for all $i$.   We next verify that $x'_{j'-1}$ remains
in the interval $[0,e]$.  Note that by our choice of interval $[j',j]$ we have  $j'-1 \in J$ while
$j' \not\in J$.  If we had $a_{j'} = -1$ then (\ref{eq:keyequation})
for $i = j'$
implies $c_{j'} > c_{j'+f}$, and according to the definition of
allowability we must have $x_{j'-1} > 0$; hence $x'_{j'-1}$ remains non-negative.
Similarly if $a_{j'} = 1$ we still have $x'_{j'-1} \le e$.   This
completes the verification.  

Since
$x'_{i} = x_{i}$ for $i \not\in [j'-1,j'+f-1]$, we in fact have $x'_i \in
[0,e]$ for all $i \in [j'-f,j'-1]$.   But  $x_i'
+ x'_{i+f} = e$ for all $i,$ and one of the two summands always lies in
$[0,e]$; therefore so does the
other.  Since $J' = \{0,\ldots,2f-1\}$ the list
$x'_0,\ldots,x'_{2f-1}$ is allowable and we
are done.  

Henceforth suppose that $j-j' < f-1$.  Consider the
following two ``moves'':
\begin{enumerate}
\item Set $J' = J \cup \{j',\ldots,j\}$, and define
  \begin{itemize}
  \item $x'_i = x_i - (c_i - c_{i+f}) - pa_i + a_{i+1}$ if $i \in
    [j',j)$,

  \item $x'_j = x_j - (c_j - c_{j+f}) - pa_j$,

  \item $x'_{j'-1} = x_{j'-1} + a_{j'}$,

 \item $x'_i = x_i$ for all remaining indices.
  \end{itemize}

\item Set $J' = J \setminus \{j'+f, \ldots, j+f\}$ and define
 \begin{itemize}
  \item $x'_{i+f} = x_{i+f} - (c_i - c_{i+f}) - pa_i + a_{i+1}$ if $i \in
    [j',j)$,

  \item $x'_{j+f} = x_{j+f} - (c_j - c_{j+f}) - pa_j$,

  \item $x'_{j'+f-1} = x_{j'+f-1} + a_{j'}$,

 \item $x'_i = x_i$ for all remaining indices.
  \end{itemize}
\end{enumerate}
In either case we have $\psi' = \psi$.  For $i \in [j',j)$ we have
$x'_i + x'_{i+f} = e$, from which it follows that $x'_i,x'_{i+f} \in
[0,e]$ (since at least one is in that interval); moreover $i,i+1$ are
either both in $J'$ or both not in $J'$ for $i \in [j',j)$ or
$[j'+f,j+f)$.  Thus $x'_i,x'_{i+f}$ are allowable for $i \in [j',j)$.  

To decide whether the list $x'_0,\ldots,x'_{2f-1}$ is allowable, the
only issue that remains is the allowability of $x'_{j'-1}$ and $x'_j$
after move (i), or of $x'_{j'+f-1}$ and $x'_{j+f}$ after move
(ii).  Note that $x'_{j'-1}$ and $x'_j$ are not the same object
since $j-j' < f-1$, and similarly for $x'_{j'+f-1}$ and $x'_{j+f}$ .
We will
argue that at least one of these two pairs must be allowable.

First consider move (i) and the allowability of $x'_j$.  We have $x'_j
+ x'_{j+f} = e - a_{j+1}$, so in particular $x'_j \in [-1,e+1]$.    Note that if $a_{j+1} \neq 0$ then the last term on the left-hand
side of 
$$ x_{j+1} + x_{j+f+1} - e \pm [c_{j+1} - c_{j+f+1}] = pa_{j+1} -a_{j+2}$$
must be nonzero, so $j+1$ is in $S$ or $T$.  By maximality  of $[j',j]$ we have $j+1 \not\in T$, so
$j+1 \in S$ and the sign $\pm$ must be $+$.  In particular either $a_{j+1}=1$ and $c_{j+1} >
c_{j+f+1}$, or $a_{j+1}=-1$ and $c_{j+1} < c_{j+f+1}$.

There are several conceivable ways that
$x'_j$ might be non-allowable.
\begin{itemize}
\item If $x'_j = -1$, then $x'_{j+f} = x_{j+f} = e$ and $a_{j+1} = 1$.  
We have seen that $a_{j+1}=1$ implies $j+f+1 \not\in J$ and $c_{j+1} >
c_{j+f+1}$.  But  since $j+f \in J$, under these
conditions $x_{j+f}=e$ would not have been allowable to begin with.
Thus $x'_j = -1$ cannot occur.

\item If $x'_j = e+1$, then $x'_{j+f} =x_{j+f} = 0$ and $a_{j+1} = -1$.  We have seen
that $a_{j+1} = -1$ implies $j+f+1 \not\in J$ and $c_{j+1} <
c_{j+f+1}$.  But since $j+f \in J$, under these conditions $x_{j+f} =
0$ would not have been allowable to begin with.  Thus $x'_j = e+1$
cannot occur.

\item If $x'_j = 0$ and is not allowable, then since $j \in J'$ we must have
$j+1 \not\in J'$.  By maximality of $[j',j]$ we have $j+f+1 \not\in
J'$ and $j+1 \not\in S \cup T$.  In particular $a_{j+1} = 0$ and
$x'_{j+f} = x_{j+f} = e$.  The allowability of $x_{j+f} = e$ when $j+f
\in J$, $j+f+1 \not\in J$ implies the allowability of $x'_{j} = 0$
when $j \in J'$, $j+1 \not\in J'$, a contradiction.

\item If $x'_j = e$ and is not allowable, then since $j \in J'$ we
  must have $j+1 \not\in J'$.    By maximality of $[j',j]$ we have $j+f+1 \not\in
J'$ and $j+1 \not\in S \cup T$.  In particular $a_{j+1} = 0$ by the
remarks above, and
$x'_{j+f} = x_{j+f} = 0$.  The allowability of $x_{j+f} = 0$ when $j+f
\in J$, $j+f+1 \not\in J$ implies the allowability of $x'_{j} = e$
when $j \in J'$, $j+1 \not\in J'$, a contradiction.
\end{itemize}

We deduce that \emph{in all cases, $x'_j$ is allowable after move (i)}.  By an
identical argument, \emph{in all cases $x'_{j+f}$ is allowable after
  move (ii).}

Now consider move (i) and the allowability of $x'_{j'-1} = x_{j'-1} +
a_{j'}$.  Note that if $a_{j'} \neq 0$ then the last term on the
left-hand side of 
$$ x_{j'} + x_{j'+f} - e - (c_{j'} - c_{j'+f}) = pa_{j'}  - a_{j'+1} $$
must be positive if $a_{j'} = 1$ and negative if $a_{j'} = -1$.  That
is, if $a_{j'} = 1$ then $c_{j'} < c_{j'+f}$ and if $a_{j'} = -1$ then
$c_{j'} > c_{j'+f}$.  There are again several conceivable ways that
$x'_{j'-1}$ might be non-allowable.
\begin{itemize}
\item  If $x'_{j'-1} = -1$, then $a_{j'} = -1$ and $x_{j'-1} = 0$.  We obtain
  $c_{j'} > c_{j'+f}$.  Since $j' \not\in J$, if we had $j'-1 \in J$ it would
  contradict the allowability of $x_{j'-1} = 0$.  So in this case we
  must have had $j'-1 \not\in J$ to begin with.

\item If $x'_{j'-1} =0$ or $e$, then since $j' \in J'$, in order to be
  non-allowable we must have $j'-1 \not\in J'$, and so $j'-1 \not\in
  J$.

\item If $x'_{j'-1} = e+1$, then $a_{j'} = 1$ and $x_{j'-1} = e$.  We obtain
  $c_{j'} < c_{j'+f}$.  Since $j' \not\in J$, if we had $j'-1 \in J$ it would
  contradict the allowability of $x_{j'-1} = e$.  So yet again we must
  have had $j'-1 \not\in J$.
\end{itemize}

We deduce that \emph{in all cases, $x'_{j'-1}$ is allowable after move
  (i) provided that $j'-1 \in J$}.   By an identical argument,
\emph{in all cases, $x'_{j'+f-1}$ is allowable after move  (ii)
  provided that $j'+f-1 \not\in J$}.

By maximality of $[j',j]$, we must have either $j'-1\in J$ or $j'+f-1
\not\in J$.  Therefore at least one of moves (i) and (ii) results in an
allowable collection $J'$ and $x'_0,\ldots,x'_{2f-1}$ with $\psi' = \psi$.   After such
a move, the set $T$ for $J'$ is strictly smaller than it was for~$J$.   The result follows.  
  \end{proof}

\section{Descent data on strongly divisible
  modules and Galois types}
\label{sec:types-and-descent}
For this section only, let $F/\Qp$ be a finite extension.  Suppose that $K/F$ is a tamely
ramified Galois extension with ramification index $e(K/F)$.  Suppose
moreover that there exists a uniformiser $\pi \in \OO_K$ with
$\pi^{e(K/F)} \in L$, where $L$ is the maximal unramified extension of
$F$ contained in $K$; then $K = L(\pi)$ and $L$ contains all of the
$e(K/F)$th roots of unity.  Let $k$ denote the residue field of
$K$ (also equal to the residue field of $L$), and $K_0$ the maximal
unramified extension of $\Qp$ contained in $K$. Let $E/\Qp$ be a
finite extension.

Suppose that $\rho$ is a potentially Barsotti-Tate  representation
$G_F \rightarrow \GL_n(E)$  that becomes Barsotti-Tate over
$K$.  We assume as usual that $K_0$ embeds into the coefficients $E$.  Write $D := D_{\st,2}^{K}(\rho)$ and let $\N$ be a strongly
divisible module with descent data over $S := S_{K,\OO_E}$ contained 
in $S[1/p] \otimes_{K_0  \otimes_{\Qp} E} D$.

Let $\tau$ be the inertial Galois type of $\rho$.  Note that $\tau$
factors through $I_L/I_K \cong \Gal(K/L)$.  Since $\Gal(K/L)$ is
abelian, $\tau$ decomposes as a direct sum of $n$ characters $\chi_i :
I_L \rightarrow \OO_E^{\times}$, and we use the isomorphism $I_L/I_K
\cong \Gal(K/L)$ to identify each $\chi_i$ as a character of
$\Gal(K/L)$.

\begin{prop}\label{prop:diagonal dd} We have $\tau = \chi_1 \oplus \cdots \oplus \chi_n$ if
  and only if there is an $S$-basis $v_1,\ldots,v_n$ of $\N$ such that
the descent data acts on $\N$ via $\widehat{g} \cdot v_i = (1 \otimes \chi_i(g))
v_i$ for all $g \in \Gal(K/L)$. 
\end{prop}

\begin{proof} For each embedding $\sigma : K_0 \rightarrow E$, let
  $e_{\sigma}$ denote the corresponding idempotent in $W(k)
  \otimes_{\Zp} \OO_E$, so that $S \cong
  \oplus_{\sigma} e_{\sigma} S$ with each $e_{\sigma} S$ a local domain.
  Since $\widehat{g}$ fixes each $e_{\sigma}$, we see that
  $\widehat{g}$ acts separately on each $e_{\sigma}\N$.

Suppose we know that $\N$ has an $S$-basis $v'_1,\ldots,v'_n$ on
which $\widehat{g} \cdot v'_i = \psi_i(g) v'_i$ for some characters $\psi_i : \Gal(K/L)
\rightarrow (W(k) \otimes_{\Zp} \OO_E)^{\times}$.  The argument in the
first three paragraphs of the proof of Proposition 6.6 of
\cite{geesavitttotallyramified} proves that $D$ has a
$K_0\otimes_{\Qp} E$-basis on which $\widehat{g}$  acts via
the maps $\psi_i$.  This is enough for the `if' direction; for the `only
if' direction, recall that by definition of $\tau$, we know that $D$ also has a
$K_0\otimes_{\Qp} E$-basis on which $\widehat{g}$ acts via the maps
$1 \otimes \chi_i$.  Since $K_0 \otimes_{\Qp} E$ is not a domain it may not
quite be the case that $\psi_i = 1\otimes\chi_i$, but at least for each
$\sigma$ the multiset $\{e_{\sigma} \psi_i\}$ is equal to the multiset
$\{\chi_i\}$: that is, we may relabel $e_{\sigma} v'_1,\ldots,e_{\sigma} v'_n$
as $v_{1,\sigma},\ldots,v_{n,\sigma}$ in such a way that $\widehat{g}
\cdot v_{i.\sigma} = \chi_i(g) v_{i,\sigma}$.  If we define $v_i =
\sum_{\sigma} v_{i,\sigma}$ then $v_1,\ldots,v_n$ is the desired basis.
Thus we are reduced to the statement at the beginning of the paragraph.

In particular it is enough to show for each $\sigma$ that the free $e_{\sigma} S$-module
$e_{\sigma} \N$ of rank $n$ has a $e_{\sigma} S$-basis on which $\widehat{g}$
acts by characters $\Gal(K/L) \rightarrow \OO_E^{\times}$.  Let
$e_{\sigma} S_0$ be the subring of $e_{\sigma} S$ consisting of power series
in $u^{e(K/F)}$.  Observe that $e_{\sigma} S$ is \emph{free} of rank
$e(K/F)$ as an $e_{\sigma} S_0$-module, with basis
$1,\ldots,u^{e(K/F)-1}$; this is because $e(K/F)$ divides the absolute
ramification index of $K$, so that if $p^{\alpha}$ exactly divides
$u^{m-1}$ in $S$ and a larger power of $p$ divides $u^{m}$, then $m$
is divisible by $e(K/F)$. 

We now regard $e_{\sigma} \N$  as a free $e_{\sigma}
S_0$-module of rank $e(K/F) n$.  Note that $\Gal(K/L)$ acts trivially on $e_{\sigma}
S_0$, so that $e_{\sigma} \N$ is actually a
$\Gal(K/L)$-representation over $e_{\sigma} S_0$.  Since $\Gal(K/L)$ is abelian and $p \nmid
\#\Gal(K/L)$, and since $\OO_E^{\times}$ contains the
$e(K/F)$th roots of unity, the module $e_{\sigma} \N$ actually has a
simultaneous $e_{\sigma} S_0$-basis of eigenvectors $y_1,\ldots,y_{e(K/F) n}$ for the
action of $\Gal(K/L)$.  Relabel the elements~$y_i$ so that
$y_1,\ldots,y_n$ are a basis for the $k_E$-vector space
$e_{\sigma}\N/(\m_{e_{\sigma} S} )e_{\sigma} \N$; here $\m_{e_{\sigma}
  S}$ is the maximal ideal of $e_{\sigma} S$ and $k_E$ is the residue
field of $E$.  By Nakayama's lemma \cite[Thm. 2.3]{matsumura}, the elements $y_1,\ldots,y_n$ are
the desired $e_{\sigma} S$-basis of $e_{\sigma} \N$.
\end{proof}

The following corollary is immediate.

\begin{cor}\label{cor:dd mod p} Let $\overline{\N}$ denote the Breuil module with descent data 
  corresponding to the $\pi_E$-torsion in the Barsotti-Tate group
  corresponding to $\rho\,|_{G_K}$.   If $\tau = \chi_1 \oplus \cdots \oplus
  \chi_n$, then $\overline{\N}$ has a $(k \otimes k_E)[u]/(u^{e'p})$-basis
  $\overline{v}_1, \ldots, \overline{v}_n$ such that $\widehat{g}
  \cdot \overline{v}_i = (1 \otimes \overline{\chi}_i(g))
  \overline{v}_i$ for all $g \in \Gal(K/L)$.  (Here $e'$ is the absolute ramification index of $K$.)
\end{cor}

Let $e_i \in k \otimes k_E$ be one of our usual idempotents.  Since
the descent data fixes $e_i$,  we see in particular that the descent data acts via $\overline{\chi}_1 \oplus
\cdots \oplus \overline{\chi}_n$ on each piece $e_i \overline{\N}$ of $\overline{\N}$.

\section{Local lifts}\label{sec:local lifts}

\subsection{Lifts of certain rank two Breuil modules}

We continue to use the following notation from Sections
\ref{sec:necess-cond-notat} to \ref{ss:irreducible}.  Let  $L$ be the unramified
quadratic extension of $F_{\p}$, and $K$ the splitting field of
$u^{p^{2f}-1}-\pi_{\p}$ over $L$.  Let $\varpi$ be a choice of
$\pi_{\p}^{1/(p^{2f}-1)}$ in $K$.  Let $k_0$ and $k$ denote the
residue fields of  $F_{\p}$ and $K$ respectively.  If
$g\in\Gal(K/F_{\p})$ then we define $\eta(g) = g(\varpi)/\varpi \in
W(k)$, so that $\overline{\eta}(g)$ is the image of $\eta(g)$ in $k$. Suppose that $F_{\p}$ has
absolute ramification index $e$, and write $e' = (p^{2f}-1)e$.

Let $E_0(u)$ be an Eisenstein polynomial for $\pi_{\p}$, so that
$E(u)=E_0(u^{p^{2f}-1})$ is an Eisenstein polynomial for $\varpi$.  Write $E(u) =
u^{e'} + pF(u)$; then $F(u)$ is a polynomial in $u^{p^{2f}-1}$ over $W(k_0)$ whose constant
term is a unit.

Let $E$, a finite extension of $\Qp$, denote the coefficient field for
our representations, with integer ring $\OO_E$ and maximal ideal
$\m_E$.  Enlarging $E$ if necessary, we assume that a
Galois closure of $K$ embeds into $E$.  In particular 
$E$ is ramified and $W(k)$ embeds into $E$.
Let $k_E$ denote the residue field of $E$.   Write $S = S_{K,\OO_E}$
 (notation as in \cite[Sec. 4]{sav04}).  Recall that $\phi : S
 \rightarrow S$ is the $W(k)$-semilinear, $\OO_E$-linear map sending $u
 \mapsto u^{p}$.  The group $\Gal(K/F_{\p})$ acts $W(k)$-semilinearly on $S$
 via $g \cdot u = ({\eta}(g)\otimes 1)u$.
 Set $c = \frac{1}{p} \phi(E(u))\in S^{\times}$.   Let $\varphi
\in \Gal(K/F_{\p})$ denote the element fixing $\varpi$ and acting
nontrivially on $L$, so that $\varphi^{-1} g \varphi = g^q$ for $g \in
\Gal(K/L)$.

We now define a rank two Breuil module $\overline{\N}$ over $(k \otimes
k_E)[u]/(u^{e'p})$ with descent data from $K$ to $F_{\p}$ with generators $\bar{v}$ and $\bar{w}$, as
follows.  Choose $J \subset \{0,\ldots,2f-1\}$, and set $\chi_i =
\overline{\chi}$ if $i \in J$ and $\chi_i = \overline{\chi}^q$
otherwise.  For each $i$, choose $r_i \in [0,e']$ such that $\chi_{i+1} = \chi_i
\eta_i^{r_i}$.  (This is equivalent to choosing an allowable $x_i$ for
$J$ and $\overline{\chi}$.)  Set $r'_i = e' - r_i$ and $\chi'_i =
\chi_i^q$, and note that $\chi_{i+1}' = \chi_i' \eta_i^{r_i'}$ since
each $r_i$ is divisible by $q-1$.  We define $\overline{\N}$ as follows.
\begin{itemize}
\item $\Fil^1 \overline{\N}$ is generated by  $\sum_{i=0}^{2f-1} u^{r_i} e_i
  \bar{v}$ and $\sum_{i=0}^{2f-1} u^{r'_i} e_i  \bar{w}$.  

\item $\phi_1(u^{r_i} e_i \bar{v}) = (1\otimes \gamma_i) e_{i+1}
  \bar{v}$ and $\phi_1(u^{r'_i} e_i \bar{w}) = (1\otimes \gamma'_i) e_{i+1}
  \bar{w}$. 

\item $\widehat{g}(e_i \bar{v}) = (1 \otimes \chi_i(g))(e_i \bar{v})$ and
  $\widehat{g}(e_i \bar{w}) = (1 \otimes \chi_i(g)^q)(e_i \bar{w})$ for
  $g \in \Gal(K/L)$.
\end{itemize}
Here $\gamma_i, \gamma_i' \in k_{E}^{\times}$.
Finally, we assume that one of
the following two sets of additional conditions holds: 
\begin{equation}\tag{RED}
\begin{cases} i \in J \ \text{if and only if} \ i+f \not\in J \\
\chi_{i+f} = \chi_i^q \\
 r_i = r_{i+f} \ \text{and} \ r'_i = r'_{i+f} \\
\gamma_i = \gamma_{i+f} \ \text{and} \ \gamma'_i = \gamma'_{i+f} \\
\widehat{\varphi}(\bar{v}) = \bar{v} \ \text{and} \
\widehat{\varphi}(\bar{w}) = \bar{w}
\end{cases}
\end{equation}
or
\begin{equation}\tag{IRR}
\begin{cases}  i \in J \ \text{if and only if} \ i+f \in J \\
\chi_{i+f} = \chi_i \\
 r_i  = r'_{i+f} \\
\gamma_i  = \gamma'_{i+f} \\
\widehat{\varphi}(\bar{v}) = \bar{w} \ \text{and} \
\widehat{\varphi}(\bar{w}) = \bar{v}.
\end{cases}
\end{equation}

The second line in each set of conditions is equivalent to the first.
Note that the relation $\widehat{\varphi}^{-1} \circ \widehat{g} \circ
\widehat{\varphi} = \widehat{g}^q$ holds in either case, and so the
module $\overline{\N}$ so-defined is indeed a Breuil module with
descent data from $K$ to $F_{\p}$.  Since $r'_{i+f} = e'-r_{i+f}$, in case
(IRR) the condition $r_i
= r'_{i+f}$ is equivalent to $x_i + x_{i+f} = e$.

\begin{thm}\label{thm: construction of the strongly divisible modules}
For each Breuil module $\overline{\N}$ as above, the generic fibre $\rhobar$ of  $\overline{\N}$ lifts to a parallel
potentially Barsotti-Tate representation $\rho$  with inertial
  type $\chi \oplus \chi^q$.
% Each Breuil module $\overline{\N}$ as above lifts to such that $\rho :=
%  T_{\st,2}^{F_{\p}}(\N)[1/p]$ is a parallel potentially Barsotti-Tate
%  representation with inertial
%  type $\chi \oplus \chi^q$.  (Since $K/F_{\p}$ is tame, the
%  determinant $\det(\rho)$  is in fact the
%  product of the cyclotomic character, a finite order character of
%  order prime to $p$, and an unramified character.)
\end{thm}

\begin{proof} We will show that $\overline{\N}$ lifts to a
 strongly divisible module $\N$ with $\OO_E$-coefficients and tame
 descent data from $K$ to $F_{\p}$ such that $\rho' :=
 T_{\st,2}^{F_{\p}}(\N)[1/p]$ is a  potentially Barsotti-Tate
representation with inertial type $\chi \oplus \chi^q$ and with all its pairs
of labeled Hodge-Tate weights equal to $\{0,1\}$.  The representation
$\rho'$ need not be parallel, but since $\det(\rho')$ is the product of the cyclotomic character, a finite order
character of order prime to $p$, and an unramified character,  we
may take $\rho$ to be the twist of $\rho'$ by a suitable unramified
character with trivial reduction mod $p$.  

Let $\N$ be the free $S$-module generated by $v$ and
  $w$; as one would imagine, $v, w$ will lift $\bar{v},\bar{w}$
  respectively.  Let $e_i \in W(k) \otimes_{\Zp} \OO_E$ also denote
  the idempotent lifting $e_i \in k \otimes_{\Fp} k_E$, and let
  $\widetilde{\sigma}_i : K_0 \hookrightarrow E$ be the embedding that
  lifts $\sigma_i$.  We define
  $\Fil^1 \N$ to be the sum of $(\Fil^1 S)\N$ and the
   submodules $\N'_i$ of $e_i \N$ defined for each $0 \le i < 2f$ as
   follows. 

If $\chi_i = \chi_{i+1}$, let $g_i h_i$ be a monic factorization of
$\widetilde{\sigma}_i (E_0(u))$
in $\OO_E[u]$ such that $\deg(g_i) = r_i/(p^{2f}-1)$ and  $\deg(h_i) =
r'_i/(p^{2f}-1)$.   This is where we
use
our hypothesis that $E$ contains all conjugates of $\pi_{\p}$ over $\Qp$.  Take $G_i = g_i(u^{p^{2f}-1})$ and $H_i =
h_i(u^{p^{2f}-1})$, so that $G_i H_i = \widetilde{\sigma}_i(E(u))$, and let $\N'_i$ be the
$S$-module generated by $(1 \otimes G_i(u)) e_i v$ and $(1 \otimes H_i(u)) e_i w$.

If $\chi_i \neq \chi_{i+1}$ let $y_i z_i$ 
be a factorization of $-p
\widetilde{\sigma}_i (F(u))$ in $\OO_E[u]$ such that $y_i \in \m_E$ and $z_i \in \m_E[u]$ or vice-versa.  This is where we use
our hypothesis that $E$ is ramified.  Take $\N'_i$ to be
the $S$-module generated by $e_i(u^{r_i} v + (1 \otimes y_i) w)$ and
$e_i((1 \otimes z_i) v + u^{r'_i} w)$.   

We impose the following extra conditions.  If $\overline{\N}$
satisfies (RED), then we insist that $G_{i+f} = G_i$, $H_{i+f} = H_i$,
$y_{i+f} = y_i$, and $z_{i+f} = z_i$; on the other hand if
$\overline{\N}$ satisfies (IRR), then we require $G_i = H_{i+f}$ and
$y_i = z_{i+f}$.   Note that this is possible
because $E(u)$ is defined over $W(k_0)$, so that
$\widetilde{\sigma}_{i+f}(E(u)) = \widetilde{\sigma}_i(E(u))$.

We define descent data from $K$ to $F_{\p}$ on $\N$ as follows,
semilinearly with respect to the action of $\Gal(K/F_{\p})$ on $S$.
Let $\widetilde{\chi}_i$ be the Teichm\"uller lift of $\chi_i$.
\begin{itemize}
\item If $g\in\Gal(K/L)$, set
$\widehat{g} (e_i v) = (1\otimes \widetilde{\chi}_i(g)) e_i v$ and $\widehat{g} (e_i w)
= (1\otimes \widetilde{\chi}_i(g)^q) e_i w$.

\item  If $\overline{\N}$ satisfies (RED) then
set $\widehat{\varphi}(v) = v$ and $\widehat{\varphi}(w) = w$.

\item If
$\overline{\N}$ satisfies (IRR) then set $\widehat{\varphi}(v) = w$
and $\widehat{\varphi}(w) = v$.
\end{itemize} In either of the last two cases, using the fact that
$\varphi$ acts trivially on $\OO_E$ and takes $e_i \mapsto e_{i+f}$,
one checks that $\widehat{\varphi}^{-1} \widehat{g} \widehat{\varphi}
= \widehat{g}^q$, so that this descent data extends to
$\Gal(K/F_{\p})$ in a well-defined way.  One checks with little
difficulty from the
definition of $\Fil^1 \N$ and the conditions on $\overline{\N}$ that
this descent data preserves $\Fil^1 \N$.  (Note in particular that
 $\Gal(K/F_{\p})$ acts trivially on $G_i,H_i,y_i,z_i$ since they are all polynomials in $u^{p^{2f}-1}$.)

Finally we wish to define a map $\phi : \N \rightarrow \N$, semilinear
with respect to $\phi$ on $S$ and such that $\phi_1 = \frac{1}{p} \phi|_{\Fil^1 \N}$ is
well-defined and satisfies
\begin{gather}
\label{phi-v-eq} \phi_1((1 \otimes G_i(u)) e_i v) = \widetilde{\gamma}_i e_{i+1} v \\
\label{phi-w-eq} \phi_1((1 \otimes H_i(u)) e_i w) = \widetilde{\gamma}'_i e_{i+1} w 
\end{gather}
if $\chi_i = \chi_{i+1}$ and 
\begin{gather}
\label{phi-v-neq} \phi_1(e_i(u^{r_i} v + (1 \otimes y_i) w)) = \widetilde{\gamma}_i
e_{i+1} v \\
\label{phi-w-neq} \phi_1(e_i((1 \otimes z_i) v + u^{r'_i} w)) = \widetilde{\gamma}'_i
 e_{i+1} w 
\end{gather}
if $\chi_i \neq \chi_{i+1}$.
Here $\widetilde{\gamma}_i, \widetilde{\gamma}'_i$ are lifts of
$\gamma_i, \gamma_i'$ to $\OO_E^{\times}$ that satisfy
$\widetilde{\gamma}_{i+f} = \widetilde{\gamma}_i$ and
$\widetilde{\gamma}'_{i+f} = \widetilde{\gamma}'_i$ in case (RED) and
$\widetilde{\gamma}'_{i+f} = \widetilde{\gamma}_i$ in case (IRR).

If $\chi_i = \chi_{i+1}$ we may satisfy \eqref{phi-v-eq} and
\eqref{phi-w-eq} by setting 
\begin{gather*} \phi(e_i v) =
c^{-1} \phi(1 \otimes H_i(u)) \widetilde{\gamma}_i e_{i+1} v \\
\phi(e_i w) =
c^{-1} \phi(1 \otimes G_i(u)) \widetilde{\gamma}'_i e_{i+1} w.\end{gather*}

If $\chi_i \neq \chi_{i+1}$, then since 
$$(E(u) \otimes 1) e_i v = u^{r'_i} (e_i(u^{r_i} v + (1 \otimes y_i) w))  - y_i
(e_i((1 \otimes z_i) v + u^{r'_i} w))$$ and similarly for $(E(u)
\otimes 1) e_i
w$, we should set
\begin{gather*} \phi(e_i v) =
c^{-1} e_{i+1} (u^{p r'_i} \widetilde{\gamma}_i v - \phi(y_i)
\widetilde{\gamma}'_i w) \\
 \phi(e_i w) =
c^{-1} e_{i+1} (u^{p r_i} \widetilde{\gamma}'_i w - \phi(z_i)
\widetilde{\gamma}_i v) .
\end{gather*}
Extending this map additively $\phi$-semilinearly to all of $\N$, one checks that
equations \eqref{phi-v-eq}--\eqref{phi-w-neq} hold, so that
$\phi(\Fil^1 \N)$ is contained in $p\N$ and generates it over $S$. One
checks futhermore that $\phi$ commutes with the descent data on $\N$
that was constructed in preceding paragraphs.  

It is now evident that $(\N, \Fil^1 \N, \phi)$ with the given descent
data is a lift of $\overline{\N}$.  It remains to check that $\N$ satisfies
the rest of the axioms of a strongly divisible module with
coefficients and descent data (namely,
conditions (2), (5)--(8), and (12) of \cite[Def. 4.1]{sav04}) and to prove
our claims about the representation $\rho'$.  

To check that $\Fil^1 \N \cap I\N = I\Fil^1 \N$ for an ideal $I
\subset \OO_E$, observe that it
suffices to check separately for each $i$ that $e_i \Fil^1 \N \cap e_i
I\N = e_i I \Fil^1 \N$.  If $\chi_i \neq \chi_{i+1}$ then this is
follows by exactly the same argument as in the proof of \cite[Thm.
6.5]{geesavitttotallyramified} (the algebraic claim being made is
literally identical).  If $\chi_i = \chi_{i+1}$ then the argument is
even easier.  Each coset in $e_i(\Fil^1 \N/(\Fil^1 S)\N)$ has a
representative of the form $e_i(a G_i v + b H_i w)$ with $a,b \in
\OO_E[u]$ such that
$\deg(a) < \deg(H_i)$ and $\deg(b) < \deg(G_i)$: terms of higher
degree can be absorbed into $(\Fil^1 S)\N$ by using the relation $E(u)
\otimes 1 = 1 \otimes G_i H_i$ in $e_iS$.  If $a G_i e_i v
+ b H_i e_i w + s_1 v + s_2 w$  lies in $e_i I\N$ (with $s_1,s_2 \in
e_i \Fil^1 S$)  then $a G_i e_i+ s_1$ must lie in $e_i I S$; the
same must be true of $a G_i e_i$ and $s_1$ individually since they
have no terms in common of the same degree in their unique expansions of the
form $\sum_{j \ge 0} q_j(u) E(u)^j/j!$ with $\deg(q_j) < \deg(E(u))$.  Then since $G_i$ is
monic the coefficients of $a$ must lie in $I$.  Similarly $s_2 \in e_i
IS$ and the coefficients of $b$ lie in $I$.  It follows that  $a G_i e_i v
+ b H_i e_i w + s_1 v + s_2 w$ actually lies in $e_i I \Fil^1 \N$. 

As for the axioms (5)--(8) and (12) of \cite[Def. 4.1]{sav04} concerning
the monodromy operator $N$, again we appeal to arguments in the
proof of \cite[Thm.
6.5]{geesavitttotallyramified}: ignoring the action of $\OO_E$ and the descent data and regarding
$(\N, \Fil^1 \N, \phi)$ simply as a strongly divisible $\Zp$-module
over $K$,
it follows from \cite[Prop. 5.1.3(1)]{bre00} that there exists
a \emph{unique} $W(k) \otimes \Zp$-endomorphism $N : \N \rightarrow \N$ satisfying
axioms (5)--(8) of \cite[Def. 4.1]{sav04}, except that we have
axiom (5) only with respect to $s \in S_{K,\Zp}$ until we know that
$N$ commutes with the action of $\OO_E$.  This commutativity, as well as
the commutativity between $N$ and the descent data (axiom (12)),
follows by uniqueness of the operator $N$.   This completes the proof
that $\N$ is a strongly divisible module.

As before, set $\rho' =  T_{\st,2}^{F_{\p}}(\N)[1/p]$, the potentially
Barsotti-Tate Galois representation associated to $\N$, and let $D:=D_{\st,2}^{K}(\rho')$.  The claim
that the Galois type of~$\rho'$ is $\chi \oplus \chi^q$ follows directly by 
Proposition \ref{prop:diagonal dd} applied with respect to the $S$-basis
$v' =
\sum_{i\in J} e_i v + \sum_{i \not\in J} e_i w$ and $w' =
\sum_{i \not\in J} e_i v + \sum_{i \in J} e_i w$ of $\N$.  

The last thing to verify is that all pairs of
labeled Hodge-Tate weights of $\rho'$ are $\{0,1\}$.  Recall that if we regard $K \otimes_{\Qp} E$
as an $S[1/p]$-algebra via the map $u \mapsto \varpi \otimes 1$, then by
\cite[Prop. 6.2.2.1]{BreuilGriffiths} there is an isomorphism 
$$f_{\varpi}: (K \otimes_{\Qp} E) \otimes_{S[1/p]} \N[1/p] \cong D_{K} := {K}
\otimes_{W(k)[1/p]} D$$
that identifies the filtrations on both sides.   It follows that
$D_{K}$ is the free $(K \otimes_{\Qp} E)$-module generated by
$f_{\varpi}(v)$ and $f_{\varpi}(w)$, and we need to show
that $\Fil^1 D_K$ is a free submodule of rank one in $D_K$.  It
suffices to check for each $i$ that $e_i \Fil^1 D_K$ is a free $e_i(K \otimes_{\Qp}
E)$-submodule of rank
one in $e_i D_K$.  If $\chi_i \neq \chi_{i+1}$ then this follows as in
the last paragraph of the proof of \cite[Prop.
6.6]{geesavitttotallyramified}: the images of $e_i(u^{r_i} v + y_iw)$ and
$e_i(z_i v + u^{r'_i}w)$ under $f_{\varpi}$ are scalar multiples of one
another, and each generates a free submodule of rank one in $e_i
\Fil^1 D_K$.

For the case $\chi_i = \chi_{i+1}$, we note that each of our
idempotents $e_i \in K \otimes_{\Qp} E$ decomposes as a sum of
idempotents $e_{\tau}$, where $\tau$ ranges over the $e(p^{2f}-1)$  embeddings $K
\hookrightarrow E$ extending $\widetilde{\sigma}_i$.   Since 
$\varpi \otimes 1 = 1 \otimes \tau(\varpi)$ in
$e_{\tau}(K\otimes_{\Qp} E)$, we deduce that the $e_{\tau}$-component
of  $f_{\varpi}(G_i(u) e_i v)$ is nonzero precisely for those $\tau$
such that the root
$\tau(\varpi)$ of $\widetilde{\sigma}_i(E(u))$ is not a root of
$G_i(u)$, and similarly the 
$e_{\tau}$-component of $f_{\varpi}(H_i(u) e_i w)$ is nonzero for
those $\tau$ such that
$\tau(\varpi)$ is not a root of $H_i(u)$.  It
follows that $e_i\Fil^1 D_K$ is free of rank one, generated by the
image of $e_i(G_i(u) v + H_i(u) w)$ under $f_{\varpi}$.
 \end{proof}

\section{An explicit description of the set of weights}\label{sec: explicit
  local description of the weights}
We maintain the notation of the previous three sections, so that $F$
is totally real and $\p|p$ is a place of $F$. The results of
sections~\ref{sec: necessity} and~\ref{sec:local lifts} may be
combined to give a complete description of when a semisimple
2-dimensional mod $p$ representation of $G_{F_\p}$ admits a parallel
potentially Barsotti-Tate lift of type $\chi\oplus\chi^{q}$ with
$\chi\neq\chi^q$. In turn, this furnishes an explicit description of
the conjectural set of weights for a  global representation
whose restriction
to each decomposition group above $p$ is semisimple.

\begin{thm}\label{thm: explicit description of when pot BT lifts exist}Write $\overline{\chi}=\prod_{i=0}^{2f-1}\eta_i^{c_i}$, with
  $0\le c_i\le p-1$. Assume
  that the local representation $\rhobar$ is semisimple. Then $\rhobar$
  has a parallel potentially Barsotti-Tate lift of type
  $\chi\oplus\chi^{q}$ if and only if one of the following three
  possibilities holds.
  \begin{enumerate}
  \item $e\ge p-1$ and $\det
\rhobar|_{I_{F_\p}}=\epsilon \cdot \overline{\chi}^{q+1}$.
\item
  $\rhobar\cong\begin{pmatrix}\psi_1&0\\0&\psi_2\end{pmatrix}$
  is decomposable, with $(\psi_1\psi_2)|_{I_{F_\p}}=\epsilon \cdot
  \overline{\chi}^{q+1}$, and \[\psi_1|_{I_{F_{\p}}}=\prod_{i\in
    J}\omega_{\pi(i)}^{c_i}\prod_{i=0}^{f-1}\omega_i^{x_i}\]
where  $J \subset
\{0,\ldots,2f-1\}$ is a subset with $i+f \in J$ if and only if $i \not\in
J$,  where the $x_i$ are allowable for $J$ and $\overline{\chi}$ (as
in Definition~\ref{defn:allowable}), and where $\pi$ is the natural projection from $\Z/2f\Z$ to $\Z/f\Z$.

% \begin{itemize}
% \item If  $i\in J$ and $i+1\in J$, or if $i\notin J$ and
% $i+1\notin J$, then we may take any $0\leq x_i\leq e$.
% \item In the other cases we allow either $0\leq x_i\leq e-1$ or $1\leq
%  x_i\leq e$. If $i\in J$ and $i+1\notin J$, then the allowable range
%  is $0\leq x_i\leq e-1$ precisely when there exists $j\geq 1$ with
%  $c_{i+k}=c_{i+k+f}$ for all $1\leq k<j$ and $c_{i+j+f}>c_{i+j}$, and
%  the situation is reversed in the case $i\notin J$ and $i+1\in J$.
% \end{itemize}

\item $\rhobar$ is irreducible, and
  $\rhobar|_{I_{F_\p}}\cong\begin{pmatrix}\psi&0\\0&\psi^q\end{pmatrix}$
  with \[\psi=\prod_{i\in
  J}\eta_{i}^{c_i}\prod_{i\notin
  J}\eta_i^{c_{i+f}}\prod_i\eta_i^{x_i}\] where $J \subset
\{0,\ldots,2f-1\}$ is a subset with $i+f \in J$ if and only if $i \in
J$, and where the $x_i$ are allowable for $J$ and $\overline{\chi}$ (as
in Definition~\ref{defn:allowable}) and satisfy $x_i + x_{i+f} =
e$.

% in addition to the following condition:
% \begin{itemize}
% \item If $i\in J$ and $i+1\notin J$, then we allow $0\leq x_i\leq e-1$
%   precisely when there is a $j\geq 1$ with $c_{i+k}=c_{i+k+f}$ for all
%   $1\leq k<j$ and $c_{i+j+f}>c_{i+j}$, and we otherwise allow $1\le
%   x_i\le e$. The situation is reversed in the case $i\notin J$ and
%   $i+1\in J$.
% \end{itemize}
  \end{enumerate}
\end{thm}
Note that by Lemma~\ref{lem:get-everything-if-e-is-large}  and the
observation following its proof, 
the condition in (iii) that the $x_i$ are
nonnegative integers satisfying $x_i + x_{i+f} = e$ may be replaced by
the condition that $0 \le x_i \le e$ for all $i$ and $\psi^{q+1}
|_{I_{F_{\p}}} = \epsilon \cdot \overline{\chi}^{q+1}$.

\begin{proof}
  The necessity of these conditions follows from  Lemma
  \ref{lem:determinant of modular of some weight} and the discussions
  of section \ref{ss:reducible} and \ref{ss:irreducible}, particularly equation
  \eqref{eq:niveau one formula} and Proposition~\ref{prop:irreduciblecase}. For their
  sufficiency, consider first the case that $\rhobar$ is
  irreducible. Then by the discussion of section
  \ref{ss:irreducible}, along with the conditions given in (iii), there is a Breuil module
  $\overline{\N}$ as in section \ref{sec:local lifts} so that the generic
fibre of $\overline{\N}$ restricted to $I_{F_{\p}}$ is
$\rhobar |_{I_{F_{\p}}}$.  Since we are in the irreducible case, the generic fibre of $\overline{\N}$ is an unramified
twist of $\rhobar$, and  the representation coming from Theorem
  \ref{thm: construction of the strongly divisible modules} applied to
  $\overline{\N}$ is an unramified twist of the desired lift of
  $\rhobar$.   This completes case (iii), and case (i) with $\rhobar$
  irreducible follows from case (iii) combined with Lemma~\ref{lem:get-everything-if-e-is-large}. 

  In the case that $\rhobar$ is reducible, note that
  case (i) will follow immediately from case (ii) and
  Proposition~\ref{prop:niveau-one-get-everything-if-e-large}.  For case (ii), observe that by the discussion in section
  \ref{ss:reducible}, the generic fibre of the rank one Breuil
  module $\M$ of  section \ref{ss:reducible}  agrees with
  $\psi_1$ up to an unramified twist; but in fact by  Corollary \ref{cor:chars
    for rank one breuil modules} the parameters $\gamma_i$ may be
  chosen so that these characters agree on the whole group
  $G_{F_{\p}}$.  Choosing the parameters $\gamma_i'$ similarly to suit
  $\psi_2$, the Breuil
  module $\M$ from  section \ref{ss:reducible} may be extended to a Breuil
  module  $\overline{\N}$ as in
  section \ref{sec:local lifts}, satisfying the conditions (RED),  whose  generic
  fibre is $\rhobar$.  The result
  again follows from Theorem \ref{thm: construction of the strongly divisible modules}.
\end{proof}

We now return to the situation where $\rhobar : G_F \to \GL_2(\Fpbar)$
is a global representation.

\begin{thm}\label{thm: explicit description of the weights in the
    semisimple case}  Suppose that $\rhobar : G_F \to \GL_2(\Fpbar)$  is
  continuous, and that $\rhobar|_{G_{F_v}}$
  is semisimple for each $v|p$. Let $\sigma=\otimes_{v|p}\sigma_v$
  be a weight. Then $\sigma\in W^?(\rhobar)$ if and only if for each
  $v|p$ we have
  \begin{enumerate}
  \item $\sigma_v$ is of type~I, and the conditions of Theorem
    \ref{thm: explicit description of when pot BT lifts exist} apply
    with $\p=v$ and $\chi=\widetilde{\sigma}_v$ (regarded as a
    character of $I_{F_v}$ by local class field theory), or
  \item $\sigma_v$ is of type~II, and \[\rhobar|_{I_{F_v}}\cong\sigma_v\begin{pmatrix}\epsilon
    & 0\\ 0 & 1\end{pmatrix}\]where $\sigma_v$ is regarded as a character of
  $I_{F_v}$ via local class field theory.
  \end{enumerate}

 In particular, if for each $v|p$ the ramification index of $F_v$ is
 at least $p-1$, and $\sigma_v$ is of type~I for each $v|p$, then
 $\sigma\in W^?(\rhobar)$ if and only if for each $v|p$ we have
 $\det\rhobar|_{I_{F_v}}=\epsilon\cdot\sigma_v^{q+1}$.
\end{thm}
\begin{proof}This all follows immediately from Definition \ref{defn:
    the set of conjectural weights} and Theorem \ref{thm: explicit
    description of when pot BT lifts exist}, except for the case
  that $\sigma_v$ is of type~II. In this case, the necessity of the
  given condition follows from Lemma \ref{lem:weight of type II
    implies cyclo}, and the sufficiency is straightforward: twisting
  reduces to the case $\sigma_v=1$, when the result follows from the
  existence of a non-crystalline extension of the trivial character by
  the cyclotomic character.
\end{proof}

\section{Proof of the weight conjecture}\label{main results}Recall that we are
assuming that $F$ is a totally real field. We now prove in many cases
that $W(\rhobar)=W^?(\rhobar)$, by combining the results of earlier sections
with the lifting machinery of Khare-Wintenberger, as interpreted by
Kisin. In particular, we use the following result.

\begin{defn}
  \label{defn:ordinary}
  Let $v|p$. We say that a representation $\rho:G_{F_v}\to\GL_2(\Qpbar)$ is
  \emph{ordinary} if $\rho|_{I_{F_v}}$ is an extension of a finite order
  character by a finite order character times the cyclotomic
  character.
\end{defn}

\begin{prop}
  \label{prop:existence-of-global-lifts}
  Suppose that $p>2$ and that $\rhobar:G_F\to\GL_2(\Fpbar)$ is
  modular. Assume that $\rhobar|_{G_{F(\zeta_p)}}$ is irreducible. If
  $p=5$ and the projective image of $\rhobar$ is isomorphic to
  $\PGL_2(\F_5)$, assume further that $[F(\zeta_p):F]=4$. Suppose that
  for each place $v|p$, $\theta_v$ is an inertial type for $I_{F_v}$
  such that $\rhobar|_{G_{F_v}}$ has a non-ordinary parallel potentially
  Barsotti-Tate lift of type $\theta_v$. Then $\rhobar$ has a modular
  lift which is parallel potentially Barsotti-Tate of type $\theta_v$ for all
  $v|p$.
\end{prop}
\begin{proof}
  This is a special case of Corollary 3.1.7 of \cite{gee061}.
\end{proof}

Recall that we defined the set of weights $W^?(\rhobar)$ conjecturally
associated to $\rhobar$ in section \ref{sec:tame lifts}, and that in
the case that the restrictions of $\rhobar$ to decomposition groups
above $p$ are semisimple, Theorem \ref{thm: explicit description of
  the weights in the semisimple case} gives an explicit description of
$W^?(\rhobar)$.

\begin{thm}\label{thm: main result} Suppose that $p>2$ and that $\rhobar:G_F\to\GL_2(\Fpbar)$ is
  modular. Assume that $\rhobar|_{G_{F(\zeta_p)}}$ is irreducible. If
  $p=5$ and the projective image of $\rhobar$ is isomorphic to
  $\PGL_2(\F_5)$, assume further that $[F(\zeta_p):F]=4$. Then
  $W(\rhobar)\subset W^?(\rhobar)$. If $\sigma\in W^?(\rhobar)$, and
  $\sigma=\otimes_{v|p}\sigma_v$ with each $\sigma_v$ of type~I, then
  $\sigma\in W(\rhobar)$. In particular, if there are no places $v|p$
  for which $\rhobar|_{G_{F_v}}$ is a twist of an extension of the
  trivial character by the cyclotomic character, then $W(\rhobar)=W^?(\rhobar)$.
\end{thm}
\begin{proof}
  The inclusion $W(\rhobar)\subset W^?(\rhobar)$ already follows from
  Lemma \ref{lem:typesversusweights}. The rest of the result follows from Lemma
  \ref{lem:typesversusweights}, Proposition
  \ref{prop:existence-of-global-lifts} and Lemma \ref{lem:weight of type II implies cyclo}, because if $\sigma_v$ is of
  type~I, any lift of type $\widetilde{\sigma}_v\oplus\widetilde{\sigma}_v^{q_v}$ only becomes crystalline over a
  nonabelian extension, and is thus certainly non-ordinary. 
\end{proof}
\begin{remark}
  It should be possible to improve this result to prove the equality
  $W^?(\rhobar)=W(\rhobar)$ under the assumption that $\rhobar$ has a
  modular lift of parallel weight two which is ordinary at any place $v$
  such that there is an element of $W^?(\rhobar|_{G_{F_v}})$ of type
  II. This would involve strengthening Proposition
  \ref{prop:existence-of-global-lifts} to include potentially
  semistable lifts, and the use of $R=T$ theorems for Hida
  families. The required results are not in the literature in the
  appropriate level of generality, however.
\end{remark}

\appendix
\section{Corrigendum to \cite{sav04}}

\numberwithin{equation}{subsection}
\newcommand{\mm}{\mathfrak{m}}
\newcommand{\Qpp}{{\mathbb Q}_{p^2}}
\newcommand{\kk}{{\mathbf k}}
\newcommand{\val}{{\rm val}}
\newcommand{\D}{{\mathcal D}}
\newcommand{\Fpp}{{\mathbb F}_{p^2}}
\newcommand{\MM}{{\mathscr M}}

The second author wishes to take this opportunity to correct an error in \cite{sav04}, as a consequence of which there is one more family of strongly divisible modules that
must be studied by the methods of \cite{sav04}. Once this is
done, the remaining claims of \cite{sav04} are unaffected. We
adopt the notation of \cite{sav04} without further comment, and
all numbered references are to that paper.

The mistake is in the statement and proof of Theorem 6.12(4).  In
the situation of that item, if $m = 1 + (p+1)j$ --- i.e., if
$i=1$ --- then the two characters $\omega_2^{m+p}$ and
$\omega_2^{pm+1}$ are both characters of niveau one, and are equal;
hence in this case the proof of Theorem 6.12(4) does \emph{not} show
that $T_{\mathrm{st},2}^{\Qp}(\MM/\mm_E) \,|_{I_p}$ decomposes as a sum of two
conjugate characters. In fact, for each choice $c$ of square root of
$\overline{w}$, the map $\MM_2' \rightarrow \MM_E(F_2/\Qpp,e_2,c,m-1)$
extends to a map $\MM' \rightarrow \MM_E(F_2/\Qp,e_2,c,j)$; by
Proposition 5.4(1), we conclude when $i=1$ that
$$ T_{st,2}^{\Qp}(\MM/\mm_E) \otimes_{\kk_E} \Fpbar
\cong \lambda_{c^{-1}} \omega^{1+j} \oplus \lambda_{-c^{-1}}
\omega^{1+j} \,.$$ This means that when $i=1$ and $\val(b)>0$ we
still need to construct a strongly divisible lattice in
$\D_{m,[1:b]}$ whose reduction mod $p$ has trivial endomorphisms;
or, conversely, we need to study deformations of type $\omt^{m} \oplus
\omt^{pm}$ (with $i=1$) of non-split residual representations of the form
$$\begin{pmatrix}
\lambda_{c^{-1}} \omega^{1+j} & * \\
0 & \lambda_{-c^{-1}} \omega^{1+j}
\end{pmatrix}\,.$$
We rectify this omission now.  Our statements are numbered to mesh
with the original article if one drops the \textbf{A.} prefix.

\setcounter{subsection}{6}
\setcounter{equation}{6}

\begin{lem} $\mathrm{(2)}$  If $i = 1$, $\val_p(b) > 0$, and $w$ is a
square in $E$, then there is $X \in S_{F_2,\OO_E}^{\times}$ satisfying
$$ X(1 \otimes wb) = 1\otimes w - \left(1 + \frac{u^{pe_2}}{p}\right)X\phi(X).$$
\end{lem}

\begin{proof} The constant term of $X$ may be taken to be $1 \otimes
  x_0$ where $x_0$ is either root of $x_0^2 + wbx_0 - w$ in
  $\OO_E^{\times}$. 
 The recursion for the coefficient $x_n$ of $u^n$ is $x_n(x_0 + wb) =
  \mathrm{lower \ terms}$, and so the recursion can be solved to
  obtain $X \in S_{F_2,\OO_E}^{\times}$.  
\end{proof}

Moreover, since $\val_p(b) > 0$, by putting the variable $B$ for $b$
we obtain an element $X_B$ of $S_{F_2,\OO_E[[B]]}$ which specializes
to $X$ under the map $\OO_E[[B]] \rightarrow \OO_E$ sending $B \mapsto b$.
Note that the image of $X$ in $(\Fpp \otimes \kk_E)[u]/u^{e_2 p}$ is
$1\otimes c$ with $c$ a square root of $\overline{w}$.   Assume
henceforth that the coefficient field $E$ contains a square root of $w$.
Now Proposition 6.10 is modified as follows.
%%% define the strongly divisible module here 

\setcounter{equation}{9}
\begin{prop}  In the case $i=1$ and $\val_p(b) > 0$, we instead define
$$ \MM_{m,[1:b]} = S_{F_2,\OO_E} \cdot g_1 + S_{F_2,\OO_E} \cdot g_2 $$
\begin{eqnarray*}
g_1 & =& \mathbf{e}_1 + \frac{X}{pw} u^{p(p-1)} \mathbf{e}_2 \\
g_2 & = & \mathbf{e}_2,
\end{eqnarray*}
and this is a strongly divisible $\OO_E$-module with descent data inside $\D_{m,[1:b]}$.
\end{prop}

\begin{proof} Put $\MM = \MM_{m,[1:b]}$.  Observe that $h := u^{p-1} g_1 +
  \left(\frac{X}{w} + (1\otimes b)\right)g_2$ lies in $\Fil^1 \MM$.
  Since $\frac{X}{w} + (1 \otimes b)$ is a unit in $S_{F_2,\OO_E}$ and
  $g_1$ does not lie in $\Fil^1 \MM$, we deduce that $\Fil^1 \MM =
  S_{F_2,\OO_E} \cdot h + (\Fil^1 S_{F_2,\OO_E})\MM$.  From this it is
  easy to check that $ I\MM \cap \Fil^1 \MM = I \Fil^1 \MM$.  Finally, we
  compute that
\begin{eqnarray*}
\phi(g_1) & = & \phi(X) u^{p^2(p-1)} g_1 + \left(1-X\phi(X)\frac{u^{pe_2}}{pw}\right) g_2 \\
\phi(g_2) & = & pwg_1 - X u^{p(p-1)} g_2
\end{eqnarray*}
both lie in $\MM$; using the defining relation for $X$ we find
$\phi_1(h) = (1 \otimes w) X^{-1} g_1 \in \MM$ and conclude that $\MM$
is a strongly divisible module.
\end{proof}

Now amend Theorem 6.12(4) so that it applies only to the case $i > 1$, and add
the following. 

\setcounter{equation}{11} 
\begin{thm}  $\mathrm{(5)}$ If $i=1$ and $\val_p(b) > 0$, then
  $T_{\mathrm{st},2}^{\Qp}(\MM/\mm_E)$ is independent of $b$ and 
$$
T_{\mathrm{st},2}^{\Qp}(\MM/\mm_E) \cong
\begin{pmatrix}
\lambda_{-c^{-1}} \omega^{1+j} & * \\
0 & \lambda_{c^{-1} } \omega^{1+j}
\end{pmatrix}
$$ with $* \neq 0$.
\end{thm}

\begin{proof} Write $\MM' = T_0(\MM/\mm_E)$.  Then $\Fil^1 \MM'$ is
  generated by  $u^{p-1} g_1 + c^{-1} g_2$ and $u^{e_2} g_1$, with
  $\phi_1(u^{p-1} g_1 + c^{-1} g_2) = cg_1$ and  $\phi_1(u^{e_2} g_1)
= u^{p^2(p-1)} cg_1 + g_2$.  Note that $\phi_1(u^{p(p-1)} g_2) = -cg_2$.  There is evidently a
 nontrivial map $\MM' \rightarrow \MM_E(F_2/\Q_p, e_2, c,j)$ sending
 $g_2 \mapsto 0$ and $g_1 \mapsto u^{p^2} \mathbf{e}$.  On the other
 hand if $f: \MM' \rightarrow \MM_E(F_2/\Qp,e_2,d,n)$ is a nontrivial map
 sending $g_1 \mapsto \alpha \mathbf{e}$ and $g_2 \mapsto \beta
 \mathbf{e}$, then $\alpha, \beta$ must both be polynomials in $u^p$
 since $g_1,g_2$ are in the image of $\phi_1$. Now if
 $\beta\neq 0$ then the relation $f\circ \phi_1 = \phi_1 \circ f$ on $u^{p(p-1)}g_2$
implies that $\beta$ is a unit times $u^p$; but then $f(u^{p-1}g_1 +
c^{-1}g_2) \in \langle u^{e_2} \mathbf{e} \rangle$ implies that
$\alpha$ has a linear term, a contradiction.  Therefore $\beta=0$, and
then it is easy to check that $c=d$ and $j=n$.  It follows that $*
\neq 0$.
\end{proof}

 (We also note the following typos in  the published version of the proof
of Theorem 6.12(4): in the first sentence, the expression
$\phi_1(u^{e_2})$ should be $\phi_1(u^{e_2} g_2)$; in the last sentence, the characters $\lambda_{c}$ should both be $\lambda_{c^{-1}}$.)

The proof of Corollary 6.15(2) should then invoke Theorem 6.12(5) in
lieu of Theorem 6.12(4) in the case of representations $\rho$ to which Theorem
6.12(5) applies, noting that the two choices for $x_0$ lead to
different reductions of $\rho$.

We now turn to deformation spaces of strongly divisible modules.  The
proof of the following proposition is identical to the proof that the
corresponding module $\MM_{m,[1:b]}$ of Proposition 6.10 is a strongly
divisible module.  As
noted in Remark 6.20, we omit the description of $N$ in the strongly
divisible module below.

\setcounter{equation}{20}   
\begin{prop} There exists a strongly divisible module with descent
  data and $\OO_E[[B]]$-coefficients as follows.
\begin{enumerate}
\setcounter{enumi}{5}
\item If $i=1$ and assuming that $w$ is a square in $E$,
$$ \MM_X = (S_{F_2,\OO_E[[B]]}) \cdot g_1 \oplus (S_{F_2,\OO_E[[B]]})\cdot g_2
,$$
$$ \Fil^1 \MM_X = S_{F_2,\OO_E[[B]]} \cdot (u^{p-1} g_1 + (w^{-1} X_B + (1
\otimes B) )g_2) + (\Fil^1 S_{F_2,\OO_E[[B]]}) \MM_X,$$
\begin{eqnarray*}
\phi(g_1) & = & \phi(X_B) u^{p^2(p-1)} g_1 + \left(1-X_B \phi(X_B )\frac{u^{pe_2}}{pw}\right) g_2  , \\
\phi(g_2) & = & pwg_1 - X_B u^{p(p-1)} g_2,
\end{eqnarray*}
$$\widehat{g}(g_1) = (\widetilde{\omega}_2^m \otimes 1)g_1, \qquad 
\widehat{g}(g_2) = (\widetilde{\omega}_2^{pm} \otimes 1)g_2.$$
\end{enumerate}
\end{prop}

Finally, one must amend the proof of Theorem 6.24 to include a proof that the
canonical injection
$$ R(2,\tau(\MM_X),\rhobar(\MM_X))_{\OO_E} \rightarrow \OO_E[[B]]$$
is a surjection; this proceeds exactly along the strategy outlined in
the proof of Theorem 6.24.  Indeed, let $\MM''$ denote the minimal
Breuil module with descent data from $F_2$ to $\Qp$ associated to the
character $\lambda_{-c^{-1}} \omega^{1+j}$, with generator
$h$ such that $\phi_1(h) = -c^{-1} h$.  
Then a map $f: \MM'' \rightarrow T_0(\MM_X/(\mm_E,B^2))$ must send
$h$ to an element of the form $\alpha u^{e_2} g_1 +
\beta(u^{p-1}g_1 + (w^{-1} X_B + B) g_2$) (where, abusing notation,
we identify elements of $S_{F_2,\OO_E[[B]]}$ with their images in
$(\Fpp\otimes \kk_E[B]/(B^2))[u]/u^{e_2 p}$).  Write $\alpha = \alpha_0 +
B \alpha_1$ and $\beta = \beta_0 + B \beta_1$ to separate out the
terms involving $B$.  The relation $f(\phi_1(h)) =
\phi_1(f(h))$ shows first that $\alpha_0 = a u^p$, $\beta_0 =
-au^{p^2}$ for some $a \in \kk_E$, by considering the relation mod
$B$; then, after some algebra, the full relation eventually implies $a=0$.
Thus the image of $f$ lies in $B \cdot T_0(\MM_X/(\mm_E,B^2))$, as desired.

\bibliographystyle{amsalpha} %Bibliography style file X.bst
\bibliography{tobybib08} % Bibliography database file Y.bib
\end{document}